\documentclass[11pt]{amsart}
\usepackage{amsmath,amssymb,latexsym,soul,cite,mathrsfs}

\usepackage{color,enumitem,graphicx}
\usepackage[colorlinks=true,urlcolor=blue,
citecolor=red,
linkcolor=blue,linktocpage,pdfpagelabels,
bookmarksnumbered,bookmarksopen]{hyperref}
\usepackage[english]{babel}

\usepackage[left=2.9cm,right=2.9cm,top=2.8cm,bottom=2.8cm]{geometry}
\usepackage[hyperpageref]{backref}

\usepackage[colorinlistoftodos]{todonotes}
\makeatletter
\providecommand\@dotsep{5}
\def\listtodoname{List of Todos}
\def\listoftodos{\@starttoc{tdo}\listtodoname}
\makeatother

\numberwithin{equation}{section}

\usepackage[utf8]{inputenc}
\usepackage{amssymb,mathtools,hyperref,cases,enumerate,xcolor}

\numberwithin{equation}{section}
\newtheorem{thm}{Theorem}[section]
\newtheorem{prop}[thm]{Proposition}
\newtheorem{lem}[thm]{Lemma}
\newtheorem{cor}[thm]{Corollary}

\theoremstyle{definition}
\newtheorem{defn}[thm]{Definition}

\newtheoremstyle{hypstyle}
{3pt}
{3pt}
{}
{}
{}
{}
{.5em}
{(\thmname{#1}$_{\thmnumber{#2}}$)}
\theoremstyle{hypstyle}
\newtheorem{hyp}{V}

\DeclarePairedDelimiter{\abs}{\lvert}{\rvert}
\DeclarePairedDelimiter{\norm}{\lVert}{\rVert}
\DeclarePairedDelimiter{\parens}{(}{)}
\DeclarePairedDelimiter{\set}{\{}{\}}
\DeclarePairedDelimiter{\brackets}{\lbrack}{\rbrack}

\DeclarePairedDelimiter{\angles}{\langle}{\rangle}

\DeclarePairedDelimiter{\coi}{\lbrack}{\lbrack}

\DeclarePairedDelimiter{\ooi}{\rbrack}{\lbrack}

\DeclareMathOperator{\cupl}{cupl}
\DeclareMathOperator{\cat}{cat}
\DeclareMathOperator{\id}{id}

\newcommand{\real}{\mathbb{R}}
\newcommand{\nat}{\mathbb{N}}

\newcommand{\Tan}{\mathrm{T}}
\newcommand{\Dif}{\mathrm{D}}
\newcommand{\dif}{\mathrm{d}}
\newcommand{\End}{\mathcal{L}}

\title[ semiclassical states for a
quasilinear Schr\"odinger-Poisson system]
{Existence and concentration of \\ semiclassical bound states for a \\
quasilinear Schr\"odinger-Poisson system}

\author[G. de Paula Ramos]{Gustavo de Paula Ramos }
\author[G. Siciliano]{Gaetano Siciliano}

\address[G. de Paula Ramos, G. Siciliano]{\newline\indent Departamento de Matem\'atica
\newline\indent 
Instituto de Matem\'atica e Estat\'istica
\newline\indent 
 Universidade de S\~ao Paulo 
\newline\indent 
Rua do Mat\~ao 1010,  05508-090, S\~ao Paulo, SP, Brazil }
\email{\href{gpramos@ime.usp.br,}{gpramos@ime.usp.br,}
\href{mailto:sicilian@ime.usp.br}{sicilian@ime.usp.br}
}


\begin{document}

\maketitle

\begin{abstract}
In the paper 
we consider the following 
 quasilinear Schr\"odinger--Poisson system in the whole space $\mathbb R^{3}$
\[
\begin{cases}
	- \varepsilon^2 \Delta u + \parens{V + \phi} u = u \abs{u}^{p - 1} \\
	- \Delta \phi - \beta \Delta_4 \phi = u^2,
\end{cases}
\]
 where $1 < p < 5, \beta > 0,V :\mathbb R^{3}\to ]0, \infty[$
and look for solutions $u,\phi:\mathbb R^{3}\to \mathbb R$ in the semiclassical regime, namely when $\varepsilon\to 0.$
By means of the Lyapunov--Schmidt method  we estimate the number of solutions by the cup-length
of the critical manifold of the external potential $V$.

\smallskip


\end{abstract}
\maketitle

\medskip

\begin{center}
\begin{minipage}{12cm}
\tableofcontents
\end{minipage}
\end{center}

\section{Introduction}

A very active topic of research in nonlinear differential equations consists in investigations about standing wave solutions to the nonlinear Schr\"odinger equation, i.e., functions $u \colon \real^3 \to \real$ that satisfy
\begin{equation} \label{NLS_epsilon}
- \varepsilon^2 \Delta u + V u = u \abs{u}^{p - 1}\quad\text{in}~\real^3,
\end{equation}
where $\varepsilon$ denotes a positive constant. In this context, we say that we are considering the \emph{semiclassical limit} when we suppose that $\varepsilon$ can be taken to be arbitrarily small -- a hypothesis that roughly describes the behavior in the interface between classical and quantum mechanics.
The function $V:\mathbb R^{3}\to \mathbb R$ is an external given potential.

In the present paper, we are interested in a slightly different version of the previous equation
where a further potential, depending  on the same $u$, is present.

More specifically, we are interested in an equation that describes
a model for the electrostatic self-interaction of electrically charged matter. 
In fact we consider the \emph{quasilinear Schr\"odinger-Poisson system} in $\mathbb R^{3}$
\begin{equation} \label{QLSP_eps} 
\begin{cases}
	- \varepsilon^2 \Delta u + \parens{V + \phi} u = u \abs{u}^{p - 1}, \quad 1<p<5 \\
	- \Delta \phi - \beta \Delta_4 \phi = u^2
\end{cases}
\end{equation}
and look for 
solutions   in the semiclassical limit, i.e. when $\varepsilon\to 0$.
As before, $V$ is an external given potential,
but now a further potential is present, which is $\phi \colon \real^3 \to \real$ and represents the electrostatic potential generated by the wave function;
 $\beta$ denotes a fixed positive constant and
\(\Delta_4 \phi := \nabla \cdot \parens{\abs{\nabla \phi}^2 \nabla \phi}\). Then the unknowns are 
$u,\phi:\mathbb R^{3}\to \mathbb R$.

When $\beta=0$ the equation for the electrostatic potential is just the Poisson equation, so the system reduces to
the well known {\sl Schr\"odinger-Poisson system} in $\mathbb R^{3}$
\begin{equation*}
\begin{cases}
	- \varepsilon^2 \Delta u + \parens{V + \phi} u = u \abs{u}^{p - 1}, \\
	- \Delta \phi = u^2
\end{cases}
\end{equation*}
where an explicit form for $\phi$ is known, namely
\[
\real^3 \ni x
\mapsto
\frac{1}{4 \pi} \int \frac{u \parens{x}^2}{\abs{x-y}} \dif y
\in
\coi{0, \infty}.
\]

The interest in studying a problem   where the second equation is an high order nonlinear
perturbation of  the classical Laplacian, relies in the fact that in some physical systems
(especially quantum mechanical models of extremely small devices in semi-conductor nanostructures)
the longitudinal field oscillations during the beam propagation has to be taken into account. 
In this case the intensity-dependent dielectric permittivity depends on the field itself and is of type 
$c_{diel}(\nabla \phi)= 1 + \varepsilon^{4}|\nabla \phi|^{2}, \varepsilon > 0$.
This model   and the corresponding equation of propagation was introduced in 
\cite{AAS} (see also \cite{MRS}).

In the mathematical literature, it seems that quasilinear Schr\"odinger-Poisson type systems
have been first addressed
 in the papers \cite{BK08, IKL, ILT}. 
However the literature is quite poor with respect to the much more studied Schr\"odinger-Poisson system.
The existing literature is essentially reduced to few papers we resume here.
In paper \cite{DLMZ} the problem with an  asymptotically linear nonlinearity 
is considered and a ground state solution is found. The problem in a bounded planar domain
is considered in \cite{FS2, PJH} with a critical and logarithmic nonlinearity.
Papers \cite{H,LY} deal with the initial boundary value problem. The zero mass case
is addressed in  \cite{WLZ, WLZ2}.
Finally we cite \cite{FS20} where the deduction of the quasilinear Schr\"odinger-Poisson
system  is done in the framework of Abelian Gauge Theories
and the case with a critical nonlinearity is studied. 
However all these papers deal with the case $\varepsilon=1$  and solutions are found 
by means of Mountain Pass type arguments and Critical Point Theory. In particular
the compactness has been recovered taking advantage of the boundedness of the domain, 
or by  exploiting the presence a  suitable parameter or  special symmetries of the problem. 

In this context, the motivation for the present paper is that, to the best of our knowledge, there are no studies on the existence of semiclassical states for \eqref{QLSP_eps},
although the Lyapunov--Schmidt method has been  successfully applied to construct semiclassical states for 
the Schr\"odinger equation 
in \cite{ABC1997, AMS01} (see also \cite[Chapter 8]{AM06}) and for the Schr\"odinger-Poisson system
 in \cite{IV08}.


%

\medskip

Coming back to our problem, note that after the change of variable $x \mapsto \varepsilon x$, \eqref{QLSP_eps}
becomes 
\begin{equation} \label{QLSP'_epsilon}
\begin{cases}
	- \Delta u + \parens{V_\varepsilon + \phi} u = u \abs{u}^{p - 1} \\
	- \varepsilon^{-2} \Delta \phi - \beta \varepsilon^{-4} \Delta_4 \phi = u^2
\end{cases}
\end{equation}
to which we will refer from now on.
We assume the following on the external potential $V \colon \real^3 \to \ooi{0, \infty}$:
\begin{hyp} [V$_1$] \label{V_1}
$V$ is of class $C^2$ with $\norm{V}_{C^2}<\infty$;	
\end{hyp}
\begin{hyp} [V$_2$] \label{V_2}
$\inf V>0$.
\end{hyp}

In order to state the main result  we need some preliminaries.
Let $\mathcal{D}^{1, 2}$ and $H^1$ respectively denote the Hilbert spaces obtained as completions of $C_c^\infty$ with respect to
\[
\angles{u_1 \mid u_2}_{\mathcal{D}^{1, 2}}
:=
\int \nabla u_1 \cdot  \nabla u_2
\quad\text{and}\quad
\angles{u_1 \mid u_2}_{H^1}
:=
\int \parens{\nabla u_1 \cdot \nabla u_2 + u_1 u_2}.
\]
For notational purposes, we also define $H^1_\varepsilon$ as the Hilbert space obtained as completion of $C_c^\infty$ with respect to
\[
\angles{u_1 \mid u_2}_{H^1_\varepsilon}
:=
\int \parens{\nabla u_1 \cdot \nabla u_2 + V_\varepsilon u_1 u_2},
\]
where
\[
V_\varepsilon \parens{x} := V \parens{\varepsilon x},
\]
being clear that $H^1_\varepsilon$ is naturally isomorphic to $H^1$ due to (\nameref{V_1}), (\nameref{V_2}).
Let also $\mathcal{D}^{1, 4}$ and $X$ denote the  Banach spaces defined as the completions of $C_c^\infty$ with respect to
\[
\norm{u}_{\mathcal{D}^{1, 4}} := \parens*{\int \abs{\nabla u}^4}^{1/4}
\quad\text{and}\quad
\norm{u}_X := \norm{u}_{\mathcal{D}^{1, 2}} + \norm{u}_{\mathcal{D}^{1, 4}}.
\]
By definition, a \emph{weak solution} of \eqref{QLSP'_epsilon} is a pair $\parens{u, \phi} \in H^1 \times X$ such that
\[
\int \left (
	 \nabla u \cdot \nabla w
	+
	\parens{V_{\varepsilon} + \phi} u w
\right)
=
\int u \abs{u}^{p-1} w
\]
and
\[
\varepsilon^{-2}\int \nabla \phi \cdot \nabla w
+
\beta\varepsilon^{-4}
\int \abs{\nabla \phi}^2 \nabla \phi \cdot \nabla w
=
\int u^2 w
\]
for every $w \in C_c^\infty$.

As we will recall in Section \ref{var_fw}, given $u\in H^{1}$ there is a unique solution $\phi_{\varepsilon}(u)\in X$
of the second equation in \eqref{QLSP'_epsilon}, and  actually there is a functional of the single variable 
$u$, $u\mapsto J_{\varepsilon}(u)$, such that its critical points give the pair $(u_{\varepsilon},\phi_{\varepsilon}(u_{\varepsilon}))$ solution of the problem.
We are then allowed to speak of solution of \eqref{QLSP'_epsilon} as referring just to 
the function $u\in H^{1}$. We will see that the functional is given by
\[
J_\varepsilon \parens{u}
:=
\frac{1}{2} \norm{u}_{H^1_\varepsilon}^2
+
\frac{3}{8} \int \phi_\varepsilon \parens{u} u^2
-
\frac{1}{8 \varepsilon^2} \norm{\phi_\varepsilon \parens{u}}_{\mathcal{D}^{1, 2}}^2
-
\frac{1}{p + 1} \norm{u}_{L^{p + 1}}^{p + 1}.
\]

It is worth noticing that
in contrast to the Schr\"odinger-Poisson system, in our case the unique solution of the second equation, 
$\phi_{\varepsilon}(u)$
 has not an explicit form, so we have to develop the necessary estimates based solely on the abstract properties of the solution operator. 

As we are studying  semiclassical limit,
 we will find solutions which are near a suitable ``particle-like'' profile function.
To make this clear, 
let
\[U \colon \real^3 \to \ooi{0, \infty}\]
be the unique positive spherically symmetric solution of the problem
\[
\begin{cases}
-\Delta u+u=u\abs{u}^{p-1}	&\text{in}~\real^3,\\
u \parens{x} \to 0			&\text{as}~\abs{x} \to \infty
\end{cases}
\]
(see the \cite[p. 23]{K89}). It follows from \cite[Proposition 4.1]{GNN81} that $U, \abs{\nabla U}$ have exponential decay at infinity; more precisely, there exists
$r \in \ooi{0, \infty}$
such that if $\abs{x} > r$, then
\[
U\parens{x}, \abs{\nabla U\parens{x}} \lesssim \abs{x}^{-1} e^{- \abs{x}}.
\]
Given $\lambda \in \ooi{0, \infty}$, it is easy to see that $\lambda^{2/\parens{p-1}} U \parens*{\lambda\, \cdot}$
is a positive solution of
\begin{equation}\label{P_lambda}
\begin{cases}
	-\Delta u+ \lambda^2 u=u\abs{u}^{p-1}	&\text{in}~\real^3,\\
	u \parens{x} \to 0			&\text{as}~\abs{x} \to \infty.
\end{cases}
\end{equation}
Also note that weak solutions in $H^1$ of \eqref{P_lambda} are precisely the critical points of the functional
$\overline{I}_\lambda \colon H^1 \to \real$
defined as
\begin{equation}\label{eq:Ibar}
\overline{I}_\lambda \parens{u}
=
\frac{1}{2} \norm{u}_{\mathcal{D}^{1, 2}}^2
+
\frac{\lambda^2}{2} \norm{u}_{L^2}^2
-
\frac{1}{p+1} \norm{u}_{L^{p+1}}^{p+1}.
\end{equation}

Now we can formalize what we understand by families of solutions  of \eqref{QLSP'_epsilon},
hence critical points of $J_{\varepsilon}$,
concentrated around points in $\real^3$.
\begin{defn}
Let $\varepsilon_0 \in \ooi{0, \infty}, x_0 \in \real^3$ and
\[
\mathcal{A}
:=
\set*{
	u_\varepsilon \in H^1: \nabla J_\varepsilon \parens{u_\varepsilon} = 0
}_{\varepsilon \in \ooi{0, \varepsilon_0}}.
\]
We say that $\mathcal{A}$ is \emph{concentrated} around $x_0$ when
\[
\norm*{
	u_\varepsilon
	-
	\lambda_{x_0}^{2 / \parens{p-1}} U \parens*{
		\frac{\lambda_{x_0}}{\varepsilon} \parens{\cdot - x_0}
	}
}_{H^1}
\to
0
\]
as $\varepsilon \to 0^+$, where $\lambda_{x_0} := V \parens{x_0}^{1 / 2}$.
\end{defn}

Let us recall also the following notions.

\begin{defn} \label{defn:crit_mfd}
Suppose that $f \in C^2 \parens{\real^3}$. We say that $\mathcal{M}$ is a \emph{non-degenerate critical manifold} of $f$ when $\mathcal{M}$ is a submanifold of $\real^3$ and given $x \in \mathcal{M}$, we have both
$\nabla f \parens{x}=0$
and
$\Tan_x \mathcal{M} = \ker \Dif^2_xf$.
\end{defn}


\begin{defn}
Suppose that $M$ is a topological space and let $\check{H}^* \parens{M}$ denotes the Alexander--Spanier cohomology of $M$. 
If $\check{H}^* \parens{M} = 0$, then the \emph{cup-length} of $M$ is defined as $\cupl \parens{M} = 1$.
Otherwise, we set
\[
	\cupl \parens{M} = \sup \set{
		k \in \nat :
		\exists \alpha_1, \ldots, \alpha_k \in \check{H}^* \parens{M};~
		\alpha_1 \cup \ldots \cup \alpha_k \neq 0
	},
\]
where $\cup$ denotes the cup product.
\end{defn}

It often holds that $\cupl \parens{M} = \cat \parens{M}$ but $\cupl \parens{M} \leq \cat \parens{M}$ in general, where $\cat$ denotes the Lusternik--Schnirelmann category. Now  we can finally state our main result.

\begin{thm}\label{thm:mult}
Suppose that $M$ is a compact non-degenerate critical manifold of $V$. 
There is  $\varepsilon_0 > 0$ such that \eqref{QLSP_eps} has at least $\cupl \parens{M} + 1$ weak solutions for any given $\varepsilon \in \ooi{0, \varepsilon_0}$. Furthermore, if
\[
\mathcal{A}
:=
\set*{
	u_n \in H^1: \nabla J_{\varepsilon_n} \parens{u_n} = 0
}_{n \in \nat}
\quad\quad
\parens*{\varepsilon_n \xrightarrow[n \to \infty]{} 0}
\]
is a family of solutions whose existence is guaranteed by the theorem, then, up to subsequence, $\mathcal{A}$ is concentrated around an $x_0 \in M$.
\end{thm}

Let us state separately the important particular case where $M=\{x_{0}\}$ is a singleton, which follows immediately from the proof of Theorem \ref{thm:mult}.
\begin{cor}
If $x_0$ is a non-degenerate critical point of $V$, then there exist
$\varepsilon_0\in\ooi{0,1}$
and a family
\[
\set*{
	u_n \in H^1: \nabla  J_{\varepsilon_n} (u_n) = 0
}_{n \in \nat}
\]
which is concentrated around $x_0$.
\end{cor}

\medskip

The paper is organized as follows. In Section \ref{var_fw}, we explore the properties of the solution operator
for the second equation in \eqref{QLSP'_epsilon} and give the variational framework and the energy functional $J_{\varepsilon}$
for the problem.
In Section \ref{sect:LS}, we perform the Lyapunov--Schmidt reduction. Finally, Theorem \ref{thm:mult} is proved in Section \ref{sect:mult}.

\subsection*{Notation}

Given functions $f, g\colon A \to \real$, 
we write $f\parens{a}\lesssim g\parens{a}$ for every $a\in A$ when there exists $C\in\ooi{0,\infty}$ such that
$\abs{f \parens{a}} \leq C \abs{g \parens{a}}$
for every $a\in A$.

The set $\set{e_1, e_2, e_3}$ will denote the canonical basis of $\real^3$. The integration domain will always  be $\real^3$ 
and we will omit, unless necessary,  the variable of integration (usually $x$) as well $\dif x$. 
We only consider functional spaces of functions defined in $\real^3$.

In general, brackets are used to enclose the argument of (multi-)linear functions. 
If  $X$, $Y$ are Banach spaces,  the set of continuous linear operators from $X$ to $Y$ 
is denoted with $\End \parens{X, Y}$; if $Y=\mathbb R$ we use the notation $X^{*}$. If $F \colon X \to Y$ is differentiable and $x \in X$, then
\[\Dif_x F : v\in X   \longmapsto  \Dif_{x}F[v] \in  Y\]
denotes the (Frechet) derivative of $F$ at $x$ evaluated in $v$. We also let $\nabla J_{\varepsilon} \parens{u} \in H^1$ denote the gradient of $J_\varepsilon$ with respect to the $H^1$-inner product computed at $u \in H^1$.

We will regard $L^{6 / 5}$ as $\parens{L^6}^*$  due to the identification of $u \in L^{6 / 5}$ with the continuous linear functional
$L^6 \ni w \mapsto \int uw \in \real$. Likewise, we often implicitly employ the Sobolev embeddings
\(
H^1, X \hookrightarrow \mathcal{D}^{1, 2} \hookrightarrow L^6
\)
and
\(
H^1 \hookrightarrow L^{12 / 5}.
\)
Other notations will be introduced whenever we need.

\medskip

\section{
The variational framework} \label{var_fw}


Let us start with the second equation of the system, namely
\begin{equation}\label{eqn:QLPois}
	- \varepsilon^{-2} \Delta \phi - \beta \varepsilon^{-4} \Delta_4 \phi = u^2,
\end{equation}
where  $u$ is given in $H^{1}$.
Following  \cite[Sections 3.1, 3.2]{FS20},  consider the more general equation
\begin{equation}\label{eqn:Tphi=g}
T_\varepsilon \parens{\phi}=g, \quad \phi \in X
\end{equation}
for a given $g \in X^*$,
where $T_\varepsilon \colon X \to X^*$ is defined as
\[
T_\varepsilon \parens{\xi} \brackets{\chi}
=
\frac{1}{\varepsilon^2} \angles{\xi \mid \chi}_{\mathcal{D}^{1, 2}}
+
\frac{\beta}{\varepsilon^4}
\int \abs{\nabla \xi}^2 \nabla \xi \cdot \nabla \chi
\]
and let us study the regularity of $T_\varepsilon$.

\begin{lem} \label{lem:regularity_T_eps}
The mapping $T_\varepsilon$ is of class $C^1$ and its derivative is given by
\begin{equation} \label{eqn:derivative_of_T_epsilon}
\Dif_\xi T_\varepsilon \brackets{\chi, \zeta}
=
\frac{1}{\varepsilon^2} \angles{\chi \mid \zeta}_{\mathcal{D}^{1, 2}}
+
\\
\frac{\beta}{\varepsilon^4} \left(
	2
	\int 
		\parens{\nabla \xi \cdot \nabla \chi}
		\parens{\nabla \xi \cdot \nabla \zeta}
	+
	\int \abs{\nabla \xi}^2 \nabla \chi \cdot \nabla \zeta
\right).
\end{equation}
\end{lem}
\begin{proof}
An elementary computation shows that $\Dif^G_\xi T_\varepsilon$, the G\^ateaux derivative of $T_\varepsilon$ in $\xi$, is given by the right-hand side in \eqref{eqn:derivative_of_T_epsilon}. 

Actually, $\Dif^G_\xi T_\varepsilon \brackets{\chi, \cdot} \colon X \to \real$ is a continuous linear functional. Indeed,
\begin{eqnarray*}
\abs*{\Dif^G_\xi T_\varepsilon \brackets{\chi, \zeta}}
&\leq&
\frac{1}{\varepsilon^2}
\norm{\chi}_{\mathcal{D}^{1, 2}} \norm{\zeta}_{\mathcal{D}^{1, 2}}
+
2
\frac{\beta}{\varepsilon^4}
\norm{\nabla \xi \cdot \nabla \chi}_{L^2}
\norm{\nabla \xi \cdot \nabla \zeta}_{L^2}
+
\frac{\beta}{\varepsilon^4}
\norm{\xi}_{\mathcal{D}^{1, 4}}^2 \norm{\nabla \chi \cdot \nabla \zeta}_{L^2}\\
&\leq&
\frac{1}{\varepsilon^2}
\parens*{
	\norm{\chi}_{\mathcal{D}^{1, 2}}
	+
	3 \frac{\beta}{\varepsilon^2} \norm{\xi}_{\mathcal{D}^{1, 4}}^2 \norm{\chi}_{\mathcal{D}^{1, 4}}
}
\norm{\zeta}_X,
\end{eqnarray*}
hence the result.

Let us prove that
$X \ni \xi \mapsto \Dif_\xi T_\varepsilon \in \End \parens{X, X^*}$
is continuous. We have
\begin{eqnarray*}
\parens{\Dif_{\xi_1} T_\varepsilon - \Dif_{\xi_2} T_\varepsilon}
\brackets{\chi, \zeta}
&=&
\frac{1}{\varepsilon^2}
\int 
	\Big(\nabla \parens{\xi_1 - \xi_2} \cdot \nabla \chi\Big)
	\parens{\nabla \xi_1 \cdot \nabla \zeta
}
\\
&+&
\frac{1}{\varepsilon^2}
\int 
	\parens{\nabla \xi_2 \cdot \nabla \chi}
	\Big(\nabla \parens{\xi_1 - \xi_2} \cdot \nabla \zeta
\Big) \\
&+&
\frac{\beta}{\varepsilon^4}
\int 
	\parens{\abs{\nabla \xi_1}^2 - \abs{\nabla \xi_2}^2}
	\nabla \chi \cdot \nabla \zeta.
\end{eqnarray*}
Therefore,
\[
\norm*{\Dif_{\xi_1} T_\varepsilon - \Dif_{\xi_2} T_\varepsilon}_{\End \parens{X, X^*}}
\leq
\frac{1}{\varepsilon^2}
\norm{\xi_1 - \xi_2}_{\mathcal{D}^{1, 4}}
\parens*{\norm{\xi_1}_{\mathcal{D}^{1, 4}} + \norm{\xi_2}_{\mathcal{D}^{1, 4}}}
+
\frac{\beta}{\varepsilon^4}
\norm*{\abs{\nabla \xi_1}^2 - \abs{\nabla \xi_2}^2}_{L^2}
\]
and so $T_\varepsilon$ is of class $C^1$.
\end{proof}

The mapping $T_\varepsilon \colon X \to X^*$ is invertible due to \cite[Theorem 5.16]{B11}, thus  let $\Phi_\varepsilon \colon X^* \to X$ denote its inverse, so that $\Phi_\varepsilon\parens{g}$ is just the unique solution to \eqref{eqn:Tphi=g}
whenever $g$ is given. 
We cannot use the Inverse Function Theorem to deduce that
$\Phi_\varepsilon$ is of class $C^1$, since
\[
X \ni \xi
\mapsto
\Dif_0 T_\varepsilon \brackets{\xi}
=
\frac{1}{\varepsilon^2}
\left.
	\angles{\xi \mid \cdot}_{\mathcal{D}^{1, 2}}
\right|_X
\in X^*
\]
does not admit a continuous inverse. Indeed, consider $\set{f_n}_{n \in \nat} \subset X$ given by
\[
f_n \parens{x}
=
\begin{cases}
\int_{1 / n}^1 t^{- 5 / 4} \dif t
&\text{if}~\abs{x} \leq 1 / n,
\\
\int_{\abs{x}}^1 t^{- 5 / 4} \dif t
&\text{if}~1 / n \leq \abs{x} \leq 1,
\\
0
&\text{if}~\abs{x} \geq 1,
\end{cases}
\]
for every $n \in \nat$ and $x \in \real^3$. It is clear that
\[
\abs{\nabla f_n \parens{x}}
=
\begin{cases}
0
&\text{if}~\abs{x} < 1 / n~\text{or}~\abs{x} > 1,
\\
\abs{x}^{- 5 / 4}
&\text{if}~1 / n < \abs{x} < 1,
\end{cases}
\]
so
$\limsup_{n \to \infty} \norm{f_n}_{\mathcal{D}^{1, 2}} < \infty$,
while
$\lim_{n \to \infty} \norm{f_n}_{\mathcal{D}^{1, 4}} = \infty$. We conclude that there cannot exist $C \in \ooi{0, \infty}$ such that
\[
\norm{\Dif_0 T_\varepsilon\brackets{f_n}}_{X^*}
=
\frac{1}{\varepsilon^2}\norm{f_n}_{\mathcal{D}^{1, 2}}
\geq
C \norm{f_n}_X
\]
for every $n \in \nat$. 
Hence the differentiability of $\Phi_\varepsilon$ is lost at $0$. However we have the following.

\begin{lem} \label{lem:Phi_epsilon_C1}
The function $\Phi_\varepsilon|_{X^* \setminus \set{0}}$ is of class $C^1$.
\end{lem}
\begin{proof}
In order to conclude from the Inverse Function Theorem, it suffices to prove that $\norm{\Dif_\xi T_\varepsilon}_{\End \parens{X, X^*}} > 0$ for every $\xi \in X \setminus \set{0}$. Indeed, it follows from \eqref{eqn:derivative_of_T_epsilon} that
\[
\norm{\Dif_\xi T_\varepsilon}_{\End \parens{X, X^*}}
\geq
\Dif_\xi T_\varepsilon
\brackets*{
	\frac{\xi}{\norm{\xi}_X},
	\frac{\xi}{\norm{\xi}_X}
}
=
\frac{1}{\norm{\xi}_X^2}
\parens*{
	\frac{1}{\varepsilon^2} \norm{\xi}_{\mathcal{D}^{1, 2}}^2
	+
	3\frac{\beta}{\varepsilon^4}
	\norm{\xi}_{\mathcal{D}^{1, 4}}^4
},
\]
hence the result.
\end{proof}

At this point, we obtain a good answer for the search of solutions to \eqref{eqn:QLPois}.

\begin{prop} \label{prop:phi_epsilon}
The \emph{solution operator},
\[
L^{12 / 5} \setminus \set{0} \ni u
\mapsto
\phi_\varepsilon \parens{u} := \Phi_\varepsilon \parens*{
	\iota^*\brackets{u^2}
}
\in
X,
\]
is of class $C^1$ and takes each $u \in L^{12 / 5} \setminus \set{0}$ to the unique weak solution of \eqref{eqn:QLPois}, where $\iota \colon X \to L^6$ denotes the mapping obtained from the composition of the Sobolev embeddings
$X \hookrightarrow \mathcal{D}^{1, 2} \hookrightarrow L^6$.
\end{prop}

Of course $\phi_{\varepsilon}(0)=0$ although this does not give a $C^1$ extension of $\phi_{\varepsilon}$ to the whole space. The following estimate on $\phi_\varepsilon$ follows immediately and will be often used throughout the paper.

\begin{lem} \label{lem:bound_phi_eps}
We have
\[
\norm{\phi_\varepsilon \parens{u}}_{\mathcal{D}^{1, 2}}
\lesssim
\varepsilon^2 \norm{u}_{L^{12 / 5}}^2
\]
for every
$\parens{\varepsilon, u} \in \ooi{0, 1}
\times
\parens{L^{12 / 5} \setminus \set{0}}$.
\end{lem}
Then we have the following  variational characterization for weak solutions of \eqref{QLSP'_epsilon}.
\begin{prop} \label{prop:variational}
The functional $J_\varepsilon \colon H^1 \setminus \set{0} \to \real$ defined as
\[
J_\varepsilon \parens{u}
:=
\frac{1}{2} \norm{u}_{H^1_\varepsilon}^2
+
\frac{3}{8} \int \phi_\varepsilon \parens{u} u^2
-
\frac{1}{8 \varepsilon^2} \norm{\phi_\varepsilon \parens{u}}_{\mathcal{D}^{1, 2}}^2
-
\frac{1}{p + 1} \norm{u}_{L^{p + 1}}^{p + 1}
\]
is of class $C^2$. Furthermore, its derivatives are given by
\[
\Dif_u J_\varepsilon \brackets{w}
=
\angles{u \mid w}_{H^1_\varepsilon}
+
\int \Big(	\phi_\varepsilon \parens{u} u w-	u \abs{u}^{p-1} w\Big),
\]
\[
\Dif^2_u J_\varepsilon \brackets{w_1, w_2}
=
\angles{w_1 \mid w_2}_{H^1_\varepsilon}
+
\int\Big(\Dif_u \phi_\varepsilon \brackets{w_1} u w_2+	\phi_\varepsilon \parens{u} w_1 w_2
-	p \abs{u}^{p-1} w_1 w_2 \Big)
\]
and the following equivalence holds:
\[
\Dif_{u} J_\varepsilon  = 0
\iff
\parens{u, \phi_\varepsilon \parens{u}}~\text{is a weak solution of}~\eqref{QLSP'_epsilon}.
\]
\end{prop}

It is worth noticing that arguing as in \cite{BK08}, a different proof gives actually that $J_{\varepsilon}\in C^{1}(H^{1})$,
and of course  $J_{\varepsilon}(0) = 0$.

\begin{proof}
It is clear that weak solutions of \eqref{QLSP'_epsilon} correspond to critical points of the functional
$\mathcal{J}_\varepsilon \colon H^1 \times X \to \real$
given by
\[
\mathcal{J}_\varepsilon \parens{u, \xi}
=
\frac{1}{2} \norm{u}_{H^1_\varepsilon}^2
+
\frac{1}{2} \int \xi u^2
-
\frac{1}{p + 1} \norm{u}_{L^{p + 1}}^{p + 1}
-
\frac{1}{2}
\parens*{
	\frac{1}{2 \varepsilon^2} \norm{\xi}_{\mathcal{D}^{1, 2}}^2
	+
	\frac{\beta}{4 \varepsilon^4} \norm{\xi}_{\mathcal{D}^{1, 4}}^4
}.
\]
It follows from Proposition \ref{prop:phi_epsilon} that
\[
\Dif_{\parens{u, \xi}} \mathcal{J}_\varepsilon \brackets{0, \cdot} = 0 \in X^*
\iff
\xi = \phi_\varepsilon \parens{u}.
\]
This leads to consider the map 
\begin{multline*}
\mathcal{J}_\varepsilon \parens{u, \phi_\varepsilon \parens{u}}
=
\frac{1}{2} \norm{u}_{H^1_\varepsilon}^2
+
\frac{1}{2} \int \phi_\varepsilon \parens{u} u^2
-
\frac{1}{p + 1} \norm{u}_{L^{p + 1}}^{p + 1}
\\
-
\frac{1}{2}
\parens*{
	\frac{1}{2 \varepsilon^2}
	\norm{\phi_\varepsilon \parens{u}}_{\mathcal{D}^{1, 2}}^2
	+
	\frac{\beta}{4 \varepsilon^4}
	\norm{\phi_\varepsilon \parens{u}}_{\mathcal{D}^{1, 4}}^4
}.
\end{multline*}
As
$
T_\varepsilon \circ \phi_\varepsilon \parens{u}
=
\iota^* \brackets{u^2}
$,
it follows that
\[
\frac{1}{\varepsilon^2} \norm{\phi_\varepsilon \parens{u}}_{\mathcal{D}^{1, 2}}^2
+
\frac{\beta}{\varepsilon^4} \norm{\phi_\varepsilon \parens{u}}_{\mathcal{D}^{1, 4}}^4
=
\int \phi_\varepsilon \parens{u} u^2
\]
and then we deduce that
$
J_\varepsilon \parens{u}
=
\mathcal{J}_\varepsilon \parens{u, \phi_\varepsilon \parens{u}}
$.

In view of the regularity of $\phi_\varepsilon$, we easily see  that $J_\varepsilon$ is of class $C^2$ 
with derivatives given above.
\end{proof}

\section{Lyapunov--Schmidt reduction} \label{sect:LS}

\subsection*{The manifold of pseudo-critical points}

We will look for weak solutions of \eqref{QLSP'_epsilon} obtained as perturbations of functions in a certain submanifold of $H^1$. The goal of this section is to introduce such a manifold and prove a crucial property of its elements.

The aforementioned submanifold of $H^1$ is given by
\[
Z_\varepsilon = \set*{
	U_{\varepsilon, z}
	:=
	\lambda_{\varepsilon z}^{2 / \parens{p - 1}}
	U \parens*{\lambda_{\varepsilon z}\parens{\cdot - z}}
	\mid
	z \in \real^3
}
\quad\text{ with }\quad
\lambda_{\varepsilon z} := V \parens{\varepsilon z}^{1 / 2},
\]
so that given $z \in \real^3$, $U_{\varepsilon, z}$ solves \eqref{P_lambda} with $\lambda = \lambda_{\varepsilon z}$.
Let us show how to use the structure of $Z_\varepsilon$ to induce a family of orthogonal decompositions of $H^1$. Given $z \in \real^3$, we define the \emph{tangent space} of $Z_\varepsilon$ at $U_{\varepsilon, z}$ as
\[
\Tan_{U_{\varepsilon, z}}Z_\varepsilon = \mathrm{span} \set*{
	\dot{U}_{\varepsilon, z, 1}, \dot{U}_{\varepsilon, z, 2}, \dot{U}_{\varepsilon, z, 3}
}
\quad \text{ where }\quad 
	\dot{U}_{\varepsilon, z, i} := \left.\frac{\dif}{\dif t} U_{\varepsilon, z+t e_i}\right|_{t=0},
\]
and we denote its $H^1$-orthogonal complement by
\begin{equation}\label{eq:We}
W_{\varepsilon, z} = \set*{w \in H^1 : \angles{w \mid u}_{H^1} = 0~\text{for every}~u \in \Tan_{U_{\varepsilon, z}}Z_\varepsilon},
\end{equation}
so that $H^1 = \Tan_{U_{\varepsilon, z}} Z_\varepsilon \oplus W_{\varepsilon, z}$.

At this point, it is important to highlight a few properties of $U_{\varepsilon, z}$ and its derivatives $\dot{U}_{\varepsilon, z, 1}$, $\dot{U}_{\varepsilon, z, 2}$, $\dot{U}_{\varepsilon, z, 3}$.
\begin{lem} \label{lem:tangent}
We have
\begin{equation} \label{eqn:aux:21}
1 \lesssim
\norm{U_{\varepsilon, z}}_{H^1},
\norm{\partial_i U_{\varepsilon, z}}_{H^1}
\lesssim
1,
\end{equation}
\begin{equation} \label{eqn:aux:22}
\norm{\dot{U}_{\varepsilon, z, i} + \partial_i U_{\varepsilon, z}}_{H^1}
\lesssim
\varepsilon,
\end{equation}
\begin{equation} \label{eqn:aux:23}
\abs*{\angles{\dot{U}_{\varepsilon, z, i} \mid \dot{U}_{\varepsilon, z, j}}_{H^1}}
\lesssim
\varepsilon
\quad\text{if}\quad
i \neq j
\end{equation}
for every $\parens{\varepsilon, z} \in \ooi{0, 1} \times \real^3$ and 
$i, j \in \set{1, 2, 3}$.
\end{lem}
\begin{proof}
The estimate \eqref{eqn:aux:21} follows from straightforward computations by taking (\nameref{V_1}), (\nameref{V_2}) into account. As for \eqref{eqn:aux:22}, it suffices to consider the $i$\textsuperscript{th} partial derivative of $z \mapsto U_{\varepsilon, z}$, Estimate \eqref{eqn:aux:21} and Hypotheses (\nameref{V_1}), (\nameref{V_2}) (for details, see \cite[p. 123]{AM06}). Finally, \eqref{eqn:aux:23} follows from \eqref{eqn:aux:22} because $\partial_i U_{\varepsilon, z}$ is odd in the $i$\textsuperscript{th} variable.
\end{proof}

The set $Z_\varepsilon$ is often called a manifold of \emph{pseudo-critical points} of $J_\varepsilon$ since we can bound the derivative of $J_\varepsilon$ at points in $Z_\varepsilon$ in term of $\nabla V \parens{\varepsilon \cdot}$. The following result formalizes this fact.

\begin{lem} \label{lem:pseudo_critical}
We have
\[
\norm{\nabla  J_\varepsilon(U_{\varepsilon, z}) }_{H^1}
\lesssim
\varepsilon \abs{\nabla V \parens{\varepsilon z}} + \varepsilon^2
\]
for every $\parens{\varepsilon, z} \in \ooi{0, 1} \times \real^3$.
\end{lem}
\begin{proof}
Recalling the expression of  $\overline I_{\lambda}$ in \eqref{eq:Ibar} and Proposition \ref{prop:variational}, it is
$$J_{\varepsilon}(u) = \overline{I}_{\lambda_{\varepsilon z}}(u)+\frac12\int \left( V_{\varepsilon} - V_{\varepsilon}(z)\right)u^{2}
+ \frac38\int\phi_{\varepsilon}(u)u^{2} -\frac{1}{8\varepsilon^{2}}\| \phi_{\varepsilon}(u)\|^{2}_{\mathcal D^{1,2}},$$
and we have 
\[
\Dif_{U_{\varepsilon, z}} J_\varepsilon \brackets{w}
=
\underbrace{
	\Dif_{U_{\varepsilon, z}} \overline{I}_{\lambda_{\varepsilon z}} \brackets{w}
}_{=0}
+
\int \left(V_\varepsilon - V_\varepsilon \parens{z}\right) U_{\varepsilon, z} w
+
\int \phi_\varepsilon \parens{U_{\varepsilon, z}} U_{\varepsilon, z} w.
\]
On one hand,
\[
\abs*{V_\varepsilon \parens{x} - V_\varepsilon \parens{z}}
\lesssim
\varepsilon \abs{\nabla V \parens{\varepsilon z}} \abs{x - z}
+
\varepsilon^2 \abs{x - z}^2
\]
for every $x \in \real^3$ due to (\nameref{V_1}). It then follows from (\nameref{V_1}), (\nameref{V_2}) and the exponential decay of $U$ at infinity that
\begin{eqnarray*}
\abs*{\int \Big(
		( V_\varepsilon - V_\varepsilon \parens{z}) U_{\varepsilon, z} w
\Big) }
&\lesssim&
\varepsilon \abs{\nabla V \parens{\varepsilon z}}
\frac{\lambda_{\varepsilon z}^{2 / \parens{p - 1}}}{\lambda_{\varepsilon z}^{5 / 2}}
\Big(
	\int \abs{x}^2 U \parens{x}	^2 \dif x
\Big)^{1 / 2}
\norm{w}_{H^1}
\\
&+&
\varepsilon^2
\frac{\lambda_{\varepsilon z}^{2 / \parens{p - 1}}}{\lambda_{\varepsilon z}^{7 / 2}}
\Big(
	\int \abs{x}^4 U \parens{x}^2 \dif x
\Big)^{1 / 2} \norm{w}_{H^1}\\
&\lesssim&
(\varepsilon \abs{\nabla V \parens{\varepsilon z}} + \varepsilon^2)
\norm{w}_{H^1}.
\end{eqnarray*}
On the other hand, we analogously have, using Lemma \ref{lem:bound_phi_eps}, that	
\begin{eqnarray*}
\int \abs*{\phi_\varepsilon \parens{U_{\varepsilon, z}} U_{\varepsilon, z} w}
&\leq&
\frac{\lambda_{\varepsilon z}^{2 / \parens{p - 1}}}{\lambda_{\varepsilon z}^2}
\norm{\phi_\varepsilon \parens{U_{\varepsilon, z}}}_{L^6}
\norm{U}_{L^{3 / 2}} \norm{w}_{L^6}\\
&\lesssim&
\norm{\phi_\varepsilon \parens{U_{\varepsilon, z}}}_{\mathcal{D}^{1, 2}}
\norm{w}_{H^1}\\
&\lesssim&\varepsilon^{2}\norm{U_{\varepsilon, z}}_{L^{12/5}}^{2} \norm{w}_{H^1}.
\end{eqnarray*}
It follows from (\nameref{V_1}), (\nameref{V_2}) that $Z_\varepsilon$ is a bounded subset of $L^{12 / 5} \setminus \set{0}$, so the result follows.
\end{proof}

\subsection*{An equivalent problem and an \emph{ansatz} for the solutions}

In view of the family of orthogonal decompositions of $H^1$ in the previous section, we can rewrite the critical point equation
\[
\nabla J_\varepsilon(u) =0, 
\quad u \in H^1
\]
as a system of two equations
\begin{numcases}
\\
\Pi_{\varepsilon, z} \brackets{\nabla   J_\varepsilon(u) }
=
0,
\label{eqn:aux}
\\
\parens{\id_{H^1} - \Pi_{\varepsilon, z}} \brackets{\nabla   J_\varepsilon (u)}
=
0,
\label{eqn:bif}
\\
u\in H^1,
\nonumber
\end{numcases}
where, recall \eqref{eq:We},  $\Pi_{\varepsilon, z} \colon H^1 \to W_{\varepsilon, z}$ is the $H^1$-orthogonal projection and we respectively name \eqref{eqn:aux}, \eqref{eqn:bif} the \emph{auxiliary} and \emph{bifurcation} equations. In this situation, the strategy for our proof will consist in looking for weak solutions to \eqref{QLSP'_epsilon} according to the \emph{ansatz}
\[
u = U_{\varepsilon, z} + w, \quad \text{where}\quad
w \in W_{\varepsilon, z} \ \  \text{and}\ \ 
\parens{\varepsilon, z} \in \ooi{0, 1} \times \real^3.
\]

\subsection*{Solving the auxiliary equation}

In this section, we proceed similarly as Ambrosetti and Malchiodi in \cite[Section 8.4]{AM06} to solve the auxiliary equation. More precisely, our main goal is to prove the lemma that follows.
\begin{lem} \label{lem:soln_aux}
There exists $\varepsilon_0 \in \ooi{0, 1}$ such that given $\varepsilon \in \ooi{0, \varepsilon_0}$, we have an application of class $C^1$,
\begin{equation} \label{eqn:map_w}
\real^3 \ni z
\mapsto
w_{\varepsilon, z} \in W_{\varepsilon, z} \subset H^1,
\end{equation}
such that
$
\Pi_{\varepsilon, z} \brackets{
	\nabla J_\varepsilon \parens{U_{\varepsilon, z} + w_{\varepsilon, z}}
} = 0
$
for every $z \in \real^3$. 
Moreover,
\[
\norm{w_{\varepsilon, z}}_{H^1}
\lesssim
\varepsilon \abs{\nabla V \parens{\varepsilon z}} + \varepsilon^2
\quad\text{and}\quad
\norm{\dot{w}_{\varepsilon, z, i}}_{H^1}
\lesssim
\left(\varepsilon \abs{\nabla V \parens{\varepsilon z}} + \varepsilon^2\right)^\mu
\]
for every $i \in \set{1, 2, 3}$ and
$\parens{\varepsilon, z} \in \ooi{0, \varepsilon_0} \times \real^3$,
where $\mu := \min \parens{1, p - 1} > 0$.
\end{lem}

Let us develop the several preliminaries needed to prove Lemma \ref{lem:soln_aux}. The next result follows by differentiating
\[
L^{12 / 5} \setminus \set{0} \ni u
\mapsto
T_\varepsilon \circ \phi_\varepsilon \parens{u}
=
\iota^* \brackets{u^2}
\in
X^*
\]
and considering the case
$
\parens{\xi, \chi}
=
\parens{\phi_\varepsilon \parens{u}, \Dif_u \phi_\varepsilon \brackets{w}}
$
in \eqref{eqn:derivative_of_T_epsilon}.
\begin{lem} \label{lem:Dif_phi_eps}
Given $\parens{u, w, \zeta} \in \parens{L^{12 / 5} \setminus \set{0}} \times L^{12 / 5} \times X$, it is
\begin{multline*}
\frac{1}{\varepsilon^2}
\angles{\Dif_u \phi_\varepsilon \brackets{w} \mid \zeta}_{\mathcal{D}^{1, 2}}
+
2\frac{\beta}{\varepsilon^4}
\int 
		\left(
			\nabla \phi_\varepsilon \parens{u}
			\cdot
			\nabla \parens{\Dif_u \phi_\varepsilon \brackets{w}}
		\right)
\left(
			\nabla \phi_\varepsilon \parens{u} \cdot \nabla \zeta\right)
+
\\
+
\frac{\beta}{\varepsilon^4}
\int 
	\abs{\nabla \phi_\varepsilon \parens{u}}^2
	\nabla \parens{\Dif_u \phi_\varepsilon \brackets{w}}
	\cdot
	\nabla \zeta
=
2 \int u w \zeta.
\end{multline*}
\end{lem}

By considering the case $\zeta = \Dif_u \phi_\varepsilon \brackets{w}$ in the previous lemma, we obtain the following estimate, reminiscent of Lemma \ref{lem:bound_phi_eps}.
\begin{cor} \label{cor:estimate_Dif_phi_eps}
We have
\[
\norm*{\Dif_u \phi_\varepsilon}_{\End \parens{L^{12 / 5}, \mathcal{D}^{1, 2}}} \lesssim \varepsilon^2 \norm{u}_{L^{12 / 5}}
\]
for every
$
\parens{\varepsilon, u}
\in
\ooi{0, 1} \times \parens{L^{12 / 5} \setminus \set {0}}
$.
\end{cor}

Now, we want to show that if $\varepsilon > 0$ is sufficiently small, then $\Dif^{2}_{U_{\varepsilon, z}} J_\varepsilon$ is coercive on a certain subspace of $H^1$. First we introduce the functional $I_\varepsilon \colon H^1 \to \real$ as given by
\begin{equation*}
I_\varepsilon \parens{u}
=
\frac{1}{2} \norm{u}_{H^1_\varepsilon}^2
-
\frac{1}{p + 1} \norm{u}_{L^{p + 1}}^{p + 1},
\end{equation*}
whose critical points  give the weak solutions of $-\Delta u +V_{\varepsilon} u = u|u|^{p-1}$, 
and then \eqref{NLS_epsilon}.

\begin{lem}\label{lem:coercive}
There exists $\varepsilon_0 \in \ooi{0, 1}$ such that
\[
\Dif^2_{U_{\varepsilon, z}} J_\varepsilon \brackets{u, u}
\gtrsim
\norm{u}_{H^1}^2
\]
for every
$\parens{\varepsilon, z} \in \ooi{0, \varepsilon_0} \times \real^3$
and
$
	u \in \parens{
		\mathrm{span}\set{U_{\varepsilon, z}}
		\oplus
		\Tan_{U_{\varepsilon, z}} \mathcal{Z}_\varepsilon
	}^\perp
$.
\end{lem}
\begin{proof}
In view of Lemma \ref{lem:bound_phi_eps} and Corollary \ref{cor:estimate_Dif_phi_eps},
\begin{align*}
\abs*{
	\Dif^2_{U_{\varepsilon, z}} J_\varepsilon \brackets{u, u}
	-
	\Dif^2_{U_{\varepsilon, z}} I_\varepsilon \brackets{u, u}
}
&=
\abs*{
	\int 
		\Dif_{U_{\varepsilon, z}} \phi_\varepsilon \brackets{u} U_{\varepsilon, z} u
		+
		\phi_\varepsilon \parens{U_{\varepsilon, z}} u^2
}
\\
&\lesssim
\norm*{\Dif_{U_{\varepsilon, z}} \phi_\varepsilon \brackets{u}}_{\mathcal{D}^{1, 2}}
\norm{u}_{H^1}
+
\norm*{\phi_\varepsilon \parens{U_{\varepsilon, z}}}_{\mathcal{D}^{1, 2}}
\norm{u}_{H^1}^2
\\
&\lesssim
\varepsilon^2
\norm{u}_{H^1}^2.
\end{align*}
Due to this estimate, the result follows from \cite[Lemma 8.9]{AM06}.
\end{proof}

Let $A_{\varepsilon, z} \colon H^1 \to H^1$ and $L_{\varepsilon, z} \colon W_{\varepsilon, z} \to W_{\varepsilon, z}$ be respectively given by
\[
A_{\varepsilon, z} \brackets{u}
:=
R \circ \Dif_{U_{\varepsilon, z}}^2 J_\varepsilon \brackets{u, \cdot}
\quad\text{and}\quad
L_{\varepsilon, z} \brackets{w}
=
\Pi_{\varepsilon, z} \circ A_{\varepsilon, z} \brackets{w},
\]
where $R \colon H^{- 1} \to H^1$ denotes the Riesz isomorphism.  The next result, which is similar to \cite[Lemma 8.10]{AM06}, establishes a sufficient condition to guarantee that $L_{\varepsilon, z}$ is invertible.

\begin{lem} \label{lem:bound_inverse_L}
There exists $\varepsilon_0 \in \ooi{0, 1}$ such that $L_{\varepsilon, z}$ is invertible and
\[\norm{L_{\varepsilon, z}^{-1}}_{\End\parens{W_{\varepsilon, z}}} \lesssim 1\]
for every
$\parens{\varepsilon, z} \in \ooi{0, \varepsilon_0} \times \real^3$.
\end{lem}
\begin{proof}
In view of Lemma \ref{lem:tangent},
\begin{equation} \label{eqn:aux:18}
\abs*{
	\angles{
		U_{\varepsilon, z} \mid \dot{U}_{\varepsilon, z, i}
	}_{H^1}
}
\leq
\abs*{
	\angles{
		U_{\varepsilon, z} \mid \dot{U}_{\varepsilon, z, i} + \partial_i U_{\varepsilon, z}
	}_{H^1}
}
+
\underbrace{
	\abs*{
		\angles{
			U_{\varepsilon, z} \mid \partial_i U_{\varepsilon, z}
		}_{H^1}
	}
}_{=0}
\lesssim
\varepsilon.
\end{equation}

We claim that
\begin{equation} \label{eqn:aux:19}
\norm{
	A_{\varepsilon, z} \brackets{U_{\varepsilon, z}}
	+
	\parens{p-1} U_{\varepsilon, z}
}_{H^1}
\lesssim
\varepsilon \abs{\nabla V \parens{\varepsilon z}} + \varepsilon^2.
\end{equation}
Indeed, as $U_{\varepsilon, z}$ solves  \eqref{P_lambda} with $\lambda=\lambda_{\varepsilon z}$, we obtain
\begin{multline*}
\Dif^2_{U_{\varepsilon, z}} J_\varepsilon \brackets{U_{\varepsilon, z}, u}
+
\parens{p - 1} \Big(
	\angles{U_{\varepsilon, z} \mid u}_{\mathcal{D}^{1, 2}}
	+
	V_{\varepsilon} \parens{ z} \angles{U_{\varepsilon, z} \mid u}_{L^2}
\Big)
=
\\
=
\int \left( V_\varepsilon - V_\varepsilon \parens{z} \right) U_{\varepsilon, z} u
+
\int \left(
	\Dif_{U_{\varepsilon, z}} \phi_\varepsilon \brackets{U_{\varepsilon, z}}
	U_{\varepsilon, z} u
	+
	\phi_\varepsilon \parens{U_{\varepsilon, z}} U_{\varepsilon, z} u
	\right)
.
\end{multline*}
The claim then follows by considering Lemma \ref{lem:bound_phi_eps}, Corollary \ref{cor:estimate_Dif_phi_eps} and arguing as in the proof of Lemma \ref{lem:pseudo_critical}.

It is clear that
\begin{multline*}
	L_{\varepsilon, z} \circ \Pi_{\varepsilon, z} \brackets{U_{\varepsilon, z}}
	=
	-\parens{p-1} \Pi_{\varepsilon, z} \brackets{U_{\varepsilon, z}}
	+
	\Pi_{\varepsilon, z} \brackets*{
		A_{\varepsilon, z} \brackets{U_{\varepsilon, z}}
		+
		\parens{p-1} U_{\varepsilon, z}
	}
	+
	\\
	+
	\Pi_{\varepsilon, z} \circ A_{\varepsilon, z} \brackets*{
		\Pi_{\varepsilon, z} \brackets{U_{\varepsilon, z}}
		-
		U_{\varepsilon, z}
	}.
\end{multline*}
Considering \eqref{eqn:aux:18} and \eqref{eqn:aux:19}, we deduce that
\[
\abs*{
	L_{\varepsilon, z} \circ \Pi_{\varepsilon, z} \brackets{U_{\varepsilon, z}}
	+
	\parens{p-1} \Pi_{\varepsilon, z} \brackets{U_{\varepsilon, z}}
}
\lesssim
\varepsilon.
\]
Therefore,
\[
\norm*{
	L_{\varepsilon, z}|_{B_{\varepsilon, z}} + \parens{p - 1} \id_{B_{\varepsilon, z}}
}_{
	\End \parens{B_{\varepsilon, z}, W_{\varepsilon, z}}
}
\lesssim \varepsilon,
\quad\text{where}\quad
B_{\varepsilon, z}
:=
\mathrm{span} \set{\Pi_{\varepsilon, z} \brackets{U_{\varepsilon,z}}}.
\]
At this point, the result follows from Lemma \ref{lem:coercive}.
\end{proof}

We need a few estimates on the derivatives of $J_\varepsilon$.

\begin{lem}\label{lem:approx_der_J}
Let $\mathcal{B}$ be a bounded subset of $H^1$ such that
\[
\set{U_{\varepsilon, z} + w: \parens{z, w} \in \real^3 \times \mathcal{B}}
\]
is bounded away from zero in $H^1$ for every $\varepsilon \in \ooi{0, 1}$. Then
\begin{equation} \label{eqn:aux:10}
\norm*{
	\Dif_{U_{\varepsilon, z} + w} J_\varepsilon - \Dif_{U_{\varepsilon, z}} J_\varepsilon
}_{H^{-1}}
\lesssim
\norm{w}_{H^1} + \norm{w}_{H^1}^p,
\end{equation}
\begin{equation} \label{eqn:aux:13}
\norm*{
	\Dif_{U_{\varepsilon, z} + w} J_\varepsilon
	-
	\Dif_{U_{\varepsilon, z}} J_\varepsilon
	-
	\Dif^2_{U_{\varepsilon, z}} J_\varepsilon \brackets{w, \cdot}
}_{H^{-1}}
\lesssim
\norm{w}_{H^1}^2 + \norm{w}_{H^1}^p + \varepsilon^2
\end{equation}
and
\begin{equation} \label{eqn:aux:9}
\norm*{
	\Dif^2_{U_{\varepsilon, z} + w} J_\varepsilon
	-
	\Dif^2_{U_{\varepsilon, z}} J_\varepsilon
}_{\End \parens{H^1, H^{-1}}}
\lesssim
\norm{w}_{H^1} + \norm{w}^{p-1}_{H^1} + \varepsilon^2
\end{equation}
for every
$\parens{\varepsilon, z, w} \in \ooi{0, 1} \times \real^3 \times \mathcal{B}$.
\end{lem}
\begin{proof}
A straightforward computation shows that
\begin{multline*}
\left(\Dif_{U_{\varepsilon, z} + w} J_\varepsilon - \Dif_{U_{\varepsilon, z}} J_\varepsilon \right)
\brackets{u}
=
\angles{w \mid u}_{H^1_\varepsilon}
+
\int 
	\left( \phi_\varepsilon \parens{U_{\varepsilon, z} + w} - \phi_\varepsilon \parens{U_{\varepsilon, z}}\right)
	u w
+
\\
-
\int 
	\left( \parens{U_{\varepsilon, z} + w} \abs{U_{\varepsilon, z} + w}^{p - 1} - U_{\varepsilon, z}^p \right) u.
\end{multline*}
We can estimate the first term by using the Cauchy--Schwarz Inequality. In view of Lemma \ref{lem:bound_phi_eps}
we infer
\begin{eqnarray*}
\int 
	\left(\phi_\varepsilon \parens{U_{\varepsilon, z} + w} - \phi_\varepsilon \parens{U_{\varepsilon, z}}\right)
	u w
&\lesssim&
\Big( 
	\norm{\phi_\varepsilon \parens{U_{\varepsilon, z} + w}}_{\mathcal{D}^{1, 2}}
	+
	\norm{\phi_\varepsilon \parens{U_{\varepsilon, z}}}_{\mathcal{D}^{1, 2}}
\Big)
\norm{u}_{H^1}
\norm{w}_{H^1}\\
&\lesssim&
\varepsilon^2 \norm{u}_{H^1} \norm{w}_{H^1}.
\end{eqnarray*}
Finally,
\[
\abs*{
	\int 
		\left(
			\parens{U_{\varepsilon, z} + w} \abs{U_{\varepsilon, z} + w}^{p - 1}
			- 
			U_{\varepsilon, z}^p
		\right)
		u
}
\lesssim
\parens*{\norm{w}_{H^1} + \norm{w}_{H^1}^p} \norm{u}_{H^1}.
\]
From these estimates we get \eqref{eqn:aux:10}.

For the second estimate, it is easy to check that
\begin{multline*}
\Big(
	\Dif_{U_{\varepsilon, z} + w} J_\varepsilon
	-
	\Dif_{U_{\varepsilon, z}} J_\varepsilon
\Big) \brackets{u}
-
\Dif^2_{U_{\varepsilon, z}} J_\varepsilon \brackets{w, u}
=
\\
=
\int \Big(
	\parens*{
		\phi_\varepsilon \parens{U_{\varepsilon, z} + w}
		-
		\phi_\varepsilon \parens{U_{\varepsilon, z}}
		-
		\Dif_{U_{\varepsilon, z}} \phi_\varepsilon \brackets{w}
	}
	U_{\varepsilon, z} u
\Big)
-
\int 
	\phi_\varepsilon \parens{U_{\varepsilon, z}} w u
+
\\
-
\int 
	\Big(
		\parens{U_{\varepsilon, z} + w} \abs{U_{\varepsilon, z} + w}^{p - 1}
		-
		U_{\varepsilon, z}^p
		-
		p U_{\varepsilon, z}^{p-1}w
	\Big) u.
\end{multline*}
Once again, we estimate the terms on the right-hand side. Due to Lemma \ref{lem:bound_phi_eps} and Corollary \ref{cor:estimate_Dif_phi_eps},
\[
\abs*{
	\int 
		\parens*{
			\phi_\varepsilon \parens{U_{\varepsilon, z} + w}
			-
			\phi_\varepsilon \parens{U_{\varepsilon, z}}
			-
			\Dif_{U_{\varepsilon, z}} \phi_\varepsilon \brackets{w}
		}
		U_{\varepsilon, z} u
	}
\lesssim
\varepsilon^2 \norm{u}_{H^1}.
\]
Furthermore,
\[
\abs*{
	\int 
	\left(
		\parens{U_{\varepsilon, z} + w} \abs{U_{\varepsilon, z} + w}^{p - 1}
		-
		U_{\varepsilon, z}^p
		-
		p U_{\varepsilon, z}^{p-1}w
	\right) u
}
\lesssim
\parens*{\norm{w}_{H^1}^2 + \norm{w}_{H^1}^p} \norm{u}_{H^1},
\]
from which we get \eqref{eqn:aux:13}.

The proof of  \eqref{eqn:aux:9} is  similar to the previous ones.
\end{proof}

By arguing similarly as in the proof of \cite[Lemma 3.4]{IV08} we prove that the auxiliary equation has solutions.

\begin{lem} \label{lem:Banach_fixed_point}
There exist
$\varepsilon_0 \in \ooi{0, 1}$ and $\bar{C} \in \ooi{0, \infty}$
such that given
$\parens{\varepsilon, z} \in \ooi{0, \varepsilon_0} \times \real^3$,
the problem
\[
\Pi_{\varepsilon, z} \brackets{
	\nabla J_\varepsilon \parens{U_{\varepsilon, z} + w}
} = 0,
\quad
w \in \mathcal{W}_{\varepsilon, z, \bar{C}}
\]
has a unique solution, where
\[
\mathcal{W}_{\varepsilon, z, \bar{C}} := \left\{
	w \in W_{\varepsilon, z}:
	\norm{w}_{H^1}
	\leq
	\bar{C} \left(
		\varepsilon \abs{\nabla V \parens{\varepsilon z}}
		+
		\varepsilon^2
	\right)
\right\}.
\]
\end{lem}
\begin{proof}
Take $\varepsilon_0 \in \ooi{0, 1}$ as furnished by Lemma \ref{lem:bound_inverse_L}. Given
$\parens{\varepsilon, z} \in \ooi{0, \varepsilon_0} \times \real^3$,
let $S_{\varepsilon, z} \colon W_{\varepsilon, z} \to W_{\varepsilon, z}$ be given by
\[
S_{\varepsilon, z} \parens{w}
=
w
-
L_{\varepsilon, z}^{-1} \circ \Pi_{\varepsilon, z} \brackets{\nabla J_\varepsilon \parens{U_{\varepsilon, z} + w}}.
\]
At this point, it suffices to show that we are in a position to use the Banach Fixed Point Theorem.

Let us prove that if $\bar{C} \in \ooi{0, \infty}$ is sufficiently small, then $S_{\varepsilon, z}|_{\mathcal{W}_{\varepsilon, z, \bar{C}}}$ is a contraction. In fact, by the definitions and Lemma \ref{lem:approx_der_J},
\begin{eqnarray*}
\Dif_{w_1} S_{\varepsilon, z} \brackets{w_2}
&=&
w_2
-
L_{\varepsilon, z}^{-1} \circ R \circ \Dif^2_{U_{\varepsilon, z} + w_1} J_\varepsilon \brackets{
	w_2, \cdot
}
\\
&=&
L_{\varepsilon, z}^{-1}
\brackets*{
	L_{\varepsilon, z} \brackets{w_2}
	-
	R \circ \Dif^2_{U_{\varepsilon, z} + w_1} J_\varepsilon \brackets{w_2, \cdot}
}\\
&\lesssim&
\parens*{
	\norm{w_1}_{H^1} + \norm{w_1}_{H^1}^{p - 1} + \varepsilon^2
}
\norm{w_2}_{H^1},
\end{eqnarray*}
hence the result.


Now, we want to show that if $\bar{C} \in \ooi{0, \infty}$ is sufficiently small, then
\[
S_{\varepsilon, z} \parens{\mathcal{W}_{\varepsilon, z, \bar{C}}}
\subset
\mathcal{W}_{\varepsilon, z, \bar{C}}.
\]
On one hand,
\[
\left \|S_{\varepsilon, z} \parens{0} \right\|_{H^1}
=
\left \| L_{\varepsilon, z}^{-1} \left[ \nabla  J_\varepsilon(U_{\varepsilon, z} ) \right]   \right\|_{H^1}
\lesssim
\norm{ \nabla  J_\varepsilon(U_{\varepsilon, z} ) }_{H^1}.
\]
On the other hand, it follows from the fact that
$S_{\varepsilon, z}|_{\mathcal{W}_{\varepsilon, z, \bar{C}}}$ 
is a contraction that
\[
\norm{S_{\varepsilon, z} \parens{0} - S_{\varepsilon, z} \parens{w}}_{H^1}
\lesssim
\norm{w}_{H^1}
\lesssim
\norm{\nabla  J_\varepsilon(U_{\varepsilon, z} ) }_{H^1}.
\]
The Triangle Inequality then implies
\[
\norm{S_{\varepsilon, z} \parens{w}}_{H^1}
\lesssim
\norm{\nabla  J_\varepsilon (U_{\varepsilon, z} ) }_{H^1}.
\]
Finally, we conclude from Lemma \ref{lem:pseudo_critical}.
\end{proof}

It follows from the definition of $Z_\varepsilon$ that
$\Dif_{U_{\varepsilon, z}} \overline{I}_{\lambda_{\varepsilon z}} = 0$
for every $z \in \real^3$. Taking the $i$\textsuperscript{th} partial derivative with respect to $z$, we obtain
\begin{equation} \label{eqn:aux:17}
\int \left(
	\nabla \dot{U}_{\varepsilon, z, i} \cdot \nabla w
	+
	V_\varepsilon \parens{z} \dot{U}_{\varepsilon, z, i} w
	+
	\varepsilon \partial_i V \parens{\varepsilon z} U_{\varepsilon, z} w
	-
	p U_{\varepsilon, z}^{p - 1} \dot{U}_{\varepsilon, z, i} w
\right)
=
0
\end{equation}
for every $w \in H^1$. The next result will follow from this observation.

\begin{lem}\label{lem:estimate_D^2J_eps}
We have
\[
\norm*{\Dif^2_{U_{\varepsilon, z}} J_\varepsilon \brackets{\dot{U}_{\varepsilon, z, i}, \cdot}}_{H^{-1}}
\lesssim
\varepsilon \abs{\nabla V \parens{\varepsilon z}} + \varepsilon^2
\]
for every $\parens{\varepsilon, z} \in \ooi{0, 1} \times \real^3$ and $i \in \set{1, 2, 3}$.
\end{lem}
\begin{proof}
Due to \eqref{eqn:aux:17}, we obtain
\begin{multline*}
\Dif^2_{U_{\varepsilon, z}} J_\varepsilon \brackets{\dot{U}_{\varepsilon, z, i}, w}
=
\int \left(
	(V_\varepsilon - V_\varepsilon \parens{z}) \dot{U}_{\varepsilon, z, i} w
	+
	\phi_\varepsilon \parens{U_{\varepsilon, z}} \dot{U}_{\varepsilon, z, i} w
\right)
+
\\
-
\varepsilon \partial_i V \parens{\varepsilon z} \int U_{\varepsilon, z} w
+
\int 
	\Dif_{U_{\varepsilon, z}} \phi_\varepsilon \brackets{
		\dot{U}_{\varepsilon, z, i}} U_{\varepsilon, z
	}
	w
.
\end{multline*}
By arguing as in the proof of Lemma \ref{lem:pseudo_critical}, we obtain
\[
\abs*{
\int \Big(
	(V_\varepsilon - V_\varepsilon \parens{z}) \dot{U}_{\varepsilon, z, i} w
	+
	\phi_\varepsilon \parens{U_{\varepsilon, z}} \dot{U}_{\varepsilon, z, i} w
\Big)
}
\lesssim
\left( \varepsilon \abs{\nabla V \parens{\varepsilon z}} + \varepsilon^2\right)
\norm{w}_{H^1}.
\]
Clearly,
\[
\varepsilon \abs*{
	\partial_i V \parens{\varepsilon z} \int U_{\varepsilon, z} w
}
\lesssim
\varepsilon \abs{\nabla V \parens{\varepsilon z}} \norm{w}_{H^1}.
\]
To finish, Corollary \ref{cor:estimate_Dif_phi_eps} implies
\[
\abs*{\int 
	\Dif_{U_{\varepsilon, z}} \phi_\varepsilon \brackets{
		\dot{U}_{\varepsilon, z, i}} U_{\varepsilon, z
	}
	w
}
\lesssim
\varepsilon^2 \norm{w}_{H^1},
\]
hence the result.
\end{proof}

Let us finally prove Lemma \ref{lem:soln_aux}.

\begin{proof}[Proof of Lemma \ref{lem:soln_aux}]
Fix $\varepsilon_0 \in \ooi{0, 1}$ for which the conclusions of Lemmas \ref{lem:bound_inverse_L}, \ref{lem:Banach_fixed_point} hold and let $\bar{C} \in \ooi{0, \infty}$ be as in Lemma \ref{lem:Banach_fixed_point}. Let
$
	\mathcal{H} \colon \mathcal{O}_{\varepsilon_0, \bar{C}} \to H^1 \times \real^3
$
be the mapping of class $C^1$ given by
\[
\mathcal{H}\parens{\varepsilon, z, w, \alpha}
=
\begin{pmatrix}
	\nabla  J_\varepsilon (U_{\varepsilon, z}+w )
	-
	\sum_{i=1}^{3} \alpha_i \dot{U}_{\varepsilon, z, i}
	\smallskip \\
	\sum_{i=1}^{3} \angles{w \mid \dot{U}_{\varepsilon, z, i}}_{H^1} e_i
\end{pmatrix},
\]
so that
\[
\Pi_{\varepsilon, z} \brackets{\nabla J_\varepsilon \parens{U_{\varepsilon, z} + w}} = 0
\quad\text{and}\quad
w \in W_{\varepsilon, z}
\]
if, and only if,
$\mathcal{H} \parens{\varepsilon, z, w, \alpha} = 0$
for a certain $\alpha \in \real^3$, where
\[
\mathcal{O}_{\varepsilon_0, \bar{C}} := \left\{
	\parens{\varepsilon, z, w, \alpha}
	\in
	\coi{0, \varepsilon_0} \times \real^3 \times H^1 \times \real^3:
	\norm{w}_{H^1}
	\leq
	\bar{C} (\varepsilon \abs{\nabla V \parens{\varepsilon z}} + \varepsilon^2)
\right\}.
\]

\emph{A preliminary result.}
Let us prove that, up to shrinking $\varepsilon_0$,
\[
H^1 \times \real^3 \ni \parens{u, \beta}
\mapsto
\Dif_{\parens{w, \alpha}} \mathcal{H}_{\varepsilon, z} \brackets{u, \beta}
:=
\Dif_{\parens{\varepsilon, z, w, \alpha}} \mathcal{H} \brackets{0, 0, u, \beta}
\in
H^1 \times \real^3
\]
is invertible and
\[
\norm*{
	\Dif_{\parens{w,\alpha}} \mathcal{H}_{\varepsilon, z} \brackets{u, \beta}
}_{H^1 \times \real^3}
\gtrsim
\norm{u}_{H^1} + \abs{\beta}
\]
for every
$
	\parens{\varepsilon, z, w, \alpha} \in \mathcal{O}_{\varepsilon_0, \bar{C}}
$
and $\parens{u, \beta} \in H^1 \times \real^3$. Indeed, it follows from Lemma \ref{lem:approx_der_J} that
\begin{eqnarray*}
\norm*{
\Dif_{\parens{w,\alpha}} \mathcal{H}_{\varepsilon, z} \brackets{u, \beta}
-
\begin{pmatrix}
	A_{\varepsilon, z} \brackets{u}
	-
	\sum_{i=1}^{3} \beta_i \dot{U}_{\varepsilon, z, i}
	\medskip \\
	\sum_{i=1}^{3} \angles{u \mid \dot{U}_{\varepsilon, z, i}}_{H^1} e_i
	\end{pmatrix}
}_{H^1 \times \real^3}
&\lesssim&
\parens*{
	\norm{w}_{H^1}
	+
	\norm{w}^{p-1}_{H^1}
	+
	\varepsilon^2
}
\norm{u}_{H^1} \\
&\lesssim &
\left(
	\varepsilon \abs{\nabla V \parens{\varepsilon z} + \varepsilon^2}
\right)^\mu
\norm{u}_{H^1}.
\end{eqnarray*}
Equivalently,
\begin{multline*}
\norm*{
	\Dif_{\parens{w,\alpha}} \mathcal{H}_{\varepsilon, z} \brackets{u, \beta}
	-
	S
	\begin{pmatrix}
		L_{\varepsilon, z} \circ \Pi_{\varepsilon, z} \brackets{u}
	\medskip	\\
		-\sum_{i=1}^{3} \beta_i \dot{U}_{\varepsilon, z, i}
	\medskip	\\
		\sum_{i=1}^{3}
		\angles{
			\parens{\id_{H^1} - \Pi_{\varepsilon, z}} \brackets{u} \mid \dot{U}_{\varepsilon, z, i}
		}_{H^1} e_i
	\end{pmatrix}
}_{H^1 \times \real^3}
\lesssim
\\
\lesssim
\Big(
	\varepsilon \abs{\nabla V \parens{\varepsilon z} + \varepsilon^2}
\Big)^\mu
\norm{u}_{H^1}
\end{multline*}
where
$S \colon H^1 \times H^1 \times \real^3 \to H^1 \times \real^3$
is defined as
$S \parens{u_1, u_2, \alpha} = \parens{u_1 + u_2, \alpha}$.
At this point, the preliminary result follows from Lemma \ref{lem:bound_inverse_L}.

\emph{The mapping \eqref{eqn:map_w}, its regularity and estimation of $\norm{w_{\varepsilon, z}}_{H^1}$.}
Considering the preliminary result and the fact that
$\mathcal{H} \parens{0, \cdot, 0, 0} \equiv 0$,
the Implicit Function Theorem furnishes an application of class $C^1$,
\[
\coi{0, \varepsilon_0} \times \real^3
\ni
\parens{\varepsilon, z}
\mapsto
\parens{w_{\varepsilon, z}, \alpha_{\varepsilon, z}}
\in
H^1 \times \real^3,
\]
such that
\begin{equation}\label{Equation:IFT}
	\mathcal{H} \parens{
		\varepsilon, z, w_{\varepsilon, z}, \alpha_{\varepsilon, z}
	}
	=
	0
\end{equation}
for every
$\parens{\varepsilon, z} \in \coi{0, \varepsilon_0} \times \real^3$
and $\alpha_{0, z} = w_{0, z} = 0$ for every $z \in \real^3$.
It follows from Lemma \ref{lem:Banach_fixed_point} that
$
\parens{\varepsilon, z, w_{\varepsilon, z}, \alpha_{\varepsilon, z}}
\in
\mathcal{O}_{\varepsilon_0, \bar{C}}
$,
hence the estimate on $\norm{w_{\varepsilon, z}}_{H^1}$.

\emph{Estimation of $\norm{\dot{w}_{\varepsilon, z, i}}_{H^1}$.} 
Taking the $i$\textsuperscript{th} partial derivative of \eqref{Equation:IFT} with respect to $z$, we deduce that
\[
\partial_i\mathcal{H}_{
	\varepsilon, w_{\varepsilon, z}, \alpha_{\varepsilon, z}
}
\parens{z}
+
\Dif_{\parens{w_{\varepsilon, z}, \alpha_{\varepsilon, z}}}
\mathcal{H}_{\varepsilon, z} \brackets{
	\dot{w}_{\varepsilon, z, i}, \dot{\alpha}_{\varepsilon, z, i}
}
=
0
\]
It follows from the preliminary result that
\begin{multline*}
	\norm{\dot{w}_{\varepsilon, z, i}}_{H^1}
	\leq
	\norm*{
		\left(
			\Dif_{\parens{w_{\varepsilon, z}, \alpha_{\varepsilon, z}}}
			\mathcal{H}_{\varepsilon, z}
		\right)^{-1}
		\brackets*{
			\partial_i\mathcal{H}_{
				\varepsilon, w_{\varepsilon, z}, \alpha_{\varepsilon, z}
			}
			\parens{z}
		}
	}_{H^1 \times \real^3}
	\lesssim
	\norm*{
		\partial_i\mathcal{H}_{
			\varepsilon, w_{\varepsilon, z}, \alpha_{\varepsilon, z}
		}
		\parens{z}
	}_{H^1 \times \real^3}.
\end{multline*}
Clearly,
\begin{eqnarray*}
	\norm*{
		\partial_i\mathcal{H}_{
			\varepsilon, w_{\varepsilon, z}, \alpha_{\varepsilon, z}
		}
		\parens{z}
	}_{H^1 \times \real^3}
	&=&
	\norm*{
		\begin{pmatrix}
			R \circ \Dif^2_{U_{\varepsilon, z}+w_{\varepsilon, z}}
			J_\varepsilon \brackets{\dot{U}_{\varepsilon, z, i}, \cdot}
			-
			\sum_{ j=1}^{3}
			\alpha_{\varepsilon, z, j} \ddot{U}_{z, i, j}
			\\
			\sum_{ j=1}^{3}
			\angles{
				w_{\varepsilon, z} \mid \ddot{U}_{z, i, j}
			}_{H^1}
			e_j
		\end{pmatrix}
	}_{H^1 \times \real^3} \\
	&\lesssim&
	\norm*{
		\Dif^2_{\parens{
			U_{\varepsilon, z}+w_{\varepsilon, z}
		}}
		J_\varepsilon \parens{\dot{U}_{\varepsilon, z, i}, \cdot}
	}_{H^{-1}}
	+
	\abs{\alpha_{\varepsilon, z}}
	+
	\norm{w_{\varepsilon, z}}_{H^1}.
\end{eqnarray*}
Lemmas \ref{lem:approx_der_J} and \ref{lem:estimate_D^2J_eps} imply
\begin{eqnarray*}
\norm*{
	\Dif^2_{U_{\varepsilon, z} + w_{\varepsilon, z}}
	J_\varepsilon \brackets{\dot{U}_{\varepsilon, z, i}, \cdot}
}_{H^{-1}}
&\leq&
\norm*{
	\Dif^2_{U_{\varepsilon, z} + w_{\varepsilon, z}}
	J_\varepsilon \brackets{\dot{U}_{\varepsilon, z, i}, \cdot}
	-
	\Dif^2_{U_{\varepsilon, z}} J_\varepsilon \brackets{\dot{U}_{\varepsilon, z, i}, \cdot}
}_{H^{-1}}
+
\norm*{
	\Dif^2_{U_{\varepsilon, z}} J_\varepsilon \brackets{\dot{U}_{\varepsilon, z, i}, \cdot}
}_{H^{-1}} \\
&\lesssim&
\left(\varepsilon \abs{\nabla V \parens{\varepsilon z}} + \varepsilon^2\right)^\mu.
\end{eqnarray*}
In view of \eqref{Equation:IFT}, Lemmas \ref{lem:pseudo_critical} and \ref{lem:tangent}, we conclude that
$
\abs{\alpha_{\varepsilon, z}}
\lesssim
\varepsilon \abs{\nabla V \parens{\varepsilon z}} + \varepsilon^2
$.
\end{proof}

\subsection*{The reduced functional}

Let $\varepsilon_0 > 0$ be furnished by Lemma \ref{lem:soln_aux} and fix $\varepsilon \in \ooi{0, \varepsilon_0}$. In this context, we can define the \emph{reduced functional}
$\widetilde{J}_\varepsilon \colon \real^3 \to \real$
as
\[
\widetilde{J}_\varepsilon \parens{z}
=
J_{\varepsilon} \parens{U_{\varepsilon, z} + w_{\varepsilon, z}}.
\]
We remark that $\widetilde{J}_\varepsilon \in C^1 \parens{\real^3}$ as a composition of mappings of class $C^1$. In fact, it suffices to look for critical points of
$\widetilde{J}_\varepsilon$
to obtain critical points of $J_\varepsilon$.
\begin{lem} \label{lem:natural_constraint}
Up to shrinking $\varepsilon_0$, the following implication holds:
\[
\nabla \widetilde{J}_\varepsilon \parens{z} = 0
\implies
\nabla J_\varepsilon \parens{U_{\varepsilon, z} + w_{\varepsilon, z}} = 0.
\]
\end{lem}
\begin{proof}
It follows from Lemma \ref{lem:soln_aux} that
\[
\nabla J_\varepsilon \parens{U_{\varepsilon, z} + w_{\varepsilon, z}}
=
\sum_{1 \leq i \leq 3}
c_{\varepsilon, z, i} \dot{U}_{\varepsilon, z, i}
\in
\Tan_{U_{\varepsilon, z} + w_{\varepsilon, z}} Z_\varepsilon.
\]
for a certain
$
c_{\varepsilon, z} := \parens{
	c_{\varepsilon, z, 1}, c_{\varepsilon, z, 2}, c_{\varepsilon, z, 3}
}
\in
\real^3
$,
and thus
\begin{align*}
\partial_i \widetilde{J}_\varepsilon \parens{z}
&=
\angles{
	\dot{U}_{\varepsilon, z, i} + \dot{w}_{\varepsilon, z, i}
	\mid
	\nabla J_\varepsilon \parens{U_{\varepsilon, z} + w_{\varepsilon, z}}
}_{H^1}
\\
&=
\sum_{1 \leq j \leq 3}
c_{\varepsilon, z, j}
\angles{
	\dot{U}_{\varepsilon, z, i} + \dot{w}_{\varepsilon, z, i}
	\mid
	\dot{U}_{z, j}
}_{H^1}
\end{align*}
for every $i \in \set{1, 2, 3}$. Equivalently,
\[
M_{\varepsilon, z}c_{\varepsilon, z}
=
0,
~\text{where}~
M_{\varepsilon, z} := \parens*{\angles{
	\dot{U}_{\varepsilon, z, i} + \dot{w}_{\varepsilon, z, i}
	\mid
	\dot{U}_{z, j}
}_{H^1}}_{1 \leq i, j \leq 3}.
\]
It follows from Lemmas \ref{lem:tangent} and \ref{lem:soln_aux} that, up to shrinking $\varepsilon_0$, $M_{\varepsilon, z}$ becomes non-singular and thus
$c_{\varepsilon, z} = 0$.
\end{proof}

\section{Proof of main result} \label{sect:mult}

We still need three preliminary results to prove Theorem \ref{thm:mult}. The first result is the proposition that follows, which is analogous to \cite[Theorem 5.1]{IV08} and may be proved accordingly.

\begin{prop} \label{prop:conc}
Suppose that $x_0 \in \real^3$, $\varepsilon_0 \in \ooi{0, 1}$ and
\[
\set*{
	u_\varepsilon \in H^1:
	\nabla J_\varepsilon \parens{u_\varepsilon} = 0
}_{\varepsilon \in \ooi{0,\varepsilon_0}}
\]
is concentrated around $x_0$. Then  $\nabla V \parens{x_0}=0$.
\end{prop}

The second result is \cite[Theorem 6.4]{C1993}, which we state below.
\begin{thm}\label{thm:approx_C1}
Let $f \in C^1 \parens{\real^3}$, $\mathcal{M}$ be a compact non-degenerate critical manifold of $f$ and $\mathcal{N}$ be a neighborhood of $\mathcal{M}$. We conclude that there exists $\delta > 0$ such that if $\norm{g}_{C^1 \parens{\mathcal{N}}} < \delta$, then $f+g$ has at least $\cupl\parens{\mathcal{M}}+1$ critical points in $\mathcal{N}$.
\end{thm}

Our last preliminary result is the following expansion of $\widetilde{J}_\varepsilon$, which is analogous to \cite[Lemma 8.11]{AM06}.
Define
\[
	C_0:=\parens*{\frac{1}{2}-\frac{1}{p+1}}\norm{U}_{L^{p+1}}^{p+1}\,,
	\quad
	\theta:=\frac{p+1}{p-1}-\frac{3}{2}.
\]

\begin{lem} \label{lem:expansion}
We have
\[
\abs*{\widetilde{J}_\varepsilon \parens{z} - C_0 V \parens{\varepsilon z}^\theta}
\lesssim
\varepsilon \abs{\nabla V \parens{\varepsilon z}} + \varepsilon^2;
\]
\[
\abs*{
	\nabla \widetilde{J}_\varepsilon \parens{z}
	-
	\varepsilon a \parens{\varepsilon z} \nabla V \parens{\varepsilon z}
}
\lesssim
\varepsilon^{1 + \mu}
\abs{\nabla V \parens{\varepsilon z}}^{1 + \mu}
+
\varepsilon^2
\]
for every
$\parens{\varepsilon,z} \in \ooi{0, \varepsilon_0} \times \real^3$,
where
$
a \parens{\varepsilon z}
:=
\theta C_0 V \parens{\varepsilon z}^{\theta - 1}
$.
\end{lem}
\begin{proof}
Let us proceed by steps.

\emph{A decomposition of $\widetilde{J}_\varepsilon$.}
The following decomposition of $\widetilde{J}_\varepsilon$ was inspired by Ambrosetti, Malchiodi and Secchi's \cite[(32)]{AMS01} and may be proved accordingly:
\[
\widetilde{J}_\varepsilon \parens{z}
=
C_0 V \parens{\varepsilon z}^\theta
+
\Lambda_\varepsilon \parens{z}
+
\Omega_\varepsilon \parens{z}
+
\Psi_\varepsilon \parens{z},
\]
where
\[
	\Lambda_\varepsilon\parens{z}
	:=
	\frac{1}{2}
	\int
		\left(
			V_\varepsilon-V_\varepsilon\parens{z}
		\right)
		U_{\varepsilon,z}^2
	+
	\int\left(
			V_\varepsilon-V_\varepsilon\parens{z}
		\right)
		U_{\varepsilon,z} w_{\varepsilon,z} \,,
\]
\[
\Omega_\varepsilon \parens{z}
:=
\frac{3}{8}
\int
	\phi_\varepsilon \parens{U_{\varepsilon,z}+w_{\varepsilon,z}}
	\parens{U_{\varepsilon,z}+w_{\varepsilon,z}}^2
-
\frac{1}{8 \varepsilon^2}
\norm{\phi_\varepsilon \parens{U_{\varepsilon,z}+w_{\varepsilon,z}}}_{\mathcal{D}^{1, 2}}^2
\]
and
\[
\Psi_\varepsilon \parens{z}
:=
\frac{1}{2} \norm{w_{\varepsilon,z}}_{H^1_\varepsilon}^2
-
\frac{1}{p + 1}
\int\Big(
	\abs{U_{\varepsilon,z}+w_{\varepsilon,z}}^{p+1}
	-
	U_{\varepsilon,z}^{p+1}
	-
	\parens{p+1} U_{\varepsilon,z}^p w_{\varepsilon,z}
\Big).
\]

\emph{Expansion of $\widetilde{J}_\varepsilon$.}
Let us estimate the terms $\Lambda_\varepsilon$, $\Omega_\varepsilon$ and $\Psi_\varepsilon$. In view of Lemma \ref{lem:soln_aux}, it suffices to argue as in the proof of Lemma \ref{lem:pseudo_critical} to obtain
\[
\abs{\Lambda_\varepsilon \parens{z}}
\lesssim
\varepsilon \abs{\nabla V \parens{\varepsilon z}} + \varepsilon^2.
\]
We know that
\[
\set{
	U_{\varepsilon, z} + w_{\varepsilon, z}
	:
	\parens{\varepsilon, z} \in \ooi{0, \varepsilon_0} \times \real^3
}
\]
is a bounded subset of $H^1 \setminus \set{0}$, so it follows from Lemma \ref{lem:bound_phi_eps} that
$\abs{\Omega_\varepsilon \parens{z}} \lesssim \varepsilon^2$. Finally,
\[
\abs{\Psi_\varepsilon \parens{z}}
\lesssim
\norm{w_{\varepsilon, z}}_{H^1}^2
+
\norm{w_{\varepsilon, z}}_{H^1}^{p + 1}
\lesssim
\Big(\varepsilon \abs{\nabla V \parens{\varepsilon z}} + \varepsilon^2\Big)^2.
\]

\emph{Expansion of $\nabla \widetilde{J}_\varepsilon$.}
By summing and subtracting terms, we obtain
\begin{eqnarray*}
\partial_i \widetilde{J}_\varepsilon \brackets{z}
&=&
\Dif_{U_{\varepsilon, z}} J_\varepsilon \brackets{
	\dot{U}_{\varepsilon, z, i}
}
+
\Dif_{U_{\varepsilon, z}} J_\varepsilon \brackets{
	\dot{w}_{\varepsilon, z, i}	
}
\\
&+&
\left(
	\Dif_{U_{\varepsilon, z} + w_{\varepsilon, z}} J_\varepsilon 
	\brackets{\dot{U}_{\varepsilon, z, i}}
	-
	\Dif_{U_{\varepsilon, z}} J_\varepsilon
	\brackets{\dot{U}_{\varepsilon, z, i}}
	-
	\Dif^2_{U_{\varepsilon, z}} J_\varepsilon
	\brackets{w_{\varepsilon,z}, \dot{U}_{\varepsilon, z, i}}
\right)
\\
&+&
\left(
	\Dif_{U_{\varepsilon, z} + w_{\varepsilon, z}} J_\varepsilon 
	\brackets{\dot{w}_{\varepsilon, z, i}}
	-
	\Dif_{U_{\varepsilon, z}} J_\varepsilon
	\brackets{\dot{w}_{\varepsilon, z, i}}
	-
	\Dif^2_{U_{\varepsilon, z}} J_\varepsilon
	\brackets{w_{\varepsilon,z}, \dot{w}_{\varepsilon, z, i}}
\right) 
\\
&+&
\Dif^2_{U_{\varepsilon, z}} J_\varepsilon
\brackets{w_{\varepsilon,z}, \dot{U}_{\varepsilon, z, i}}
+
\Dif^2_{U_{\varepsilon, z}} J_\varepsilon
\brackets{w_{\varepsilon,z}, \dot{w}_{\varepsilon, z, i}}, 
\end{eqnarray*}
and thus
\begin{eqnarray} \label{eqn:aux:14}
\abs*{
	\partial_i \widetilde{J}_\varepsilon \parens{z}
	-
	\Dif_{U_{\varepsilon, z}} J_\varepsilon \brackets{\dot{U}_{\varepsilon, z, i}}
}
&\leq&
\abs*{\Dif_{U_{\varepsilon, z}} J_\varepsilon \brackets{\dot{w}_{\varepsilon, z, i}}}
\\
&+&
\norm*{\mathcal{R}_{\varepsilon, z}}_{H^{-1}}
\parens*{
	\norm{\dot{U}_{\varepsilon, z, i}}_{H^1}
	+
	\norm{\dot{w}_{\varepsilon, z, i}}_{H^1}
}\nonumber
\\
&+&
\abs*{
	\Dif^2_{U_{\varepsilon, z}} J_\varepsilon
	\brackets{w_{\varepsilon,z}, \dot{U}_{\varepsilon, z, i}}
}
+
\abs*{
	\Dif^2_{U_{\varepsilon, z}} J_\varepsilon
	\brackets{w_{\varepsilon,z}, \dot{w}_{\varepsilon, z, i}}
}, \nonumber
\end{eqnarray}
where $\mathcal{R}_{\varepsilon, z} \colon H^1 \to \real$ is given by
\[
\mathcal{R}_{\varepsilon, z} \brackets{u}
=
\Dif_{U_{\varepsilon, z} + w_{\varepsilon, z}} J_\varepsilon \brackets{u}
-
\Dif_{U_{\varepsilon, z}} J_\varepsilon \brackets{u}
-
\Dif^2_{U_{\varepsilon, z}} J_\varepsilon \brackets{w_{\varepsilon,z}, u}.
\]

We want to estimate the terms on the right-hand side of \eqref{eqn:aux:14}. In view of Lemmas \ref{lem:pseudo_critical} and \ref{lem:soln_aux},
\[
\abs*{\Dif_{U_{\varepsilon, z}} J_\varepsilon \brackets{\dot{w}_{\varepsilon, z, i}}}
\leq
\norm*{\Dif_{U_{\varepsilon, z}} J_\varepsilon}_{H^{-1}}
\norm{\dot{w}_{\varepsilon, z, i}}_{H^1}
\lesssim
\left(
	\varepsilon \abs{\nabla V \parens{\varepsilon z}} + \varepsilon^2
\right)^{1 + \mu}.
\]
Due to Lemmas \ref{lem:tangent}, \ref{lem:soln_aux} and \ref{lem:approx_der_J},
\[
\norm*{
	\mathcal{R}_{\varepsilon, z}
}_{H^{-1}}
\parens*{
	\norm{\dot{U}_{\varepsilon, z, i}}_{H^1}
	+
	\norm{\dot{w}_{\varepsilon, z, i}}_{H^1}
}
\lesssim
\left(\varepsilon \abs{\nabla V \parens{\varepsilon z}} + \varepsilon^2\right)^{1 + \mu}
+
\varepsilon^2.
\]
It follows from Lemmas \ref{lem:soln_aux}, \ref{lem:estimate_D^2J_eps} that
\[
\abs*{
	\Dif^2_{U_{\varepsilon, z}} J_\varepsilon
	\brackets{w_{\varepsilon,z}, \dot{U}_{\varepsilon, z, i}}
}
\leq
\norm*{
	\Dif^2_{U_{\varepsilon, z}} J_\varepsilon \brackets{
		\dot{U}_{\varepsilon, z, i}, \cdot
	}
}_{H^{-1}}
\norm{w_{\varepsilon,z}}_{H^1}
\lesssim
\left(
	\varepsilon \abs{\nabla V \parens{\varepsilon z}} + \varepsilon^2
\right)^2.
\]
Similarly, Lemma \ref{lem:soln_aux} implies
\[
\abs*{
	\Dif^2_{U_{\varepsilon, z}} J_\varepsilon
	\brackets{w_{\varepsilon,z}, \dot{w}_{\varepsilon, z, i}}
}
\lesssim
\left(\varepsilon \abs{\nabla V \parens{\varepsilon z}} + \varepsilon^2\right)^{1 + \mu}.
\]
We conclude that
\begin{equation} \label{eqn:aux:15}
\abs*{
	\partial_i \widetilde{J}_\varepsilon \parens{z}
	-
	\Dif_{U_{\varepsilon, z}} J_\varepsilon \brackets{\dot{U}_{\varepsilon, z, i}}
}
\lesssim
\varepsilon^{1 + \mu}
\abs{\nabla V \parens{\varepsilon z}}^{1 + \mu}
+
\varepsilon^2.
\end{equation}

In view of \eqref{eqn:aux:15}, we  have just to prove that
\[
\abs*{
	\Dif_{U_{\varepsilon, z}} J_\varepsilon \brackets{\dot{U}_{\varepsilon, z, i}}
	-
	\varepsilon a \parens{\varepsilon z} \partial_i V \parens{\varepsilon z}
}
\lesssim
\varepsilon^2.
\]
Indeed, by taking the $i$\textsuperscript{th} partial derivative with respect to $z$ of
\[
J_\varepsilon \parens{U_{\varepsilon, z}}
=
C_0 V \parens{\varepsilon z}^\theta
+
\frac{1}{2} \int
	(V_\varepsilon - V \parens{\varepsilon z})
	U_{\varepsilon, z}^2
+
\frac{3}{8}
\int \phi_\varepsilon \parens{U_{\varepsilon, z}} U_{\varepsilon, z}^2
-
\frac{1}{8 \varepsilon^2}
\norm{\phi_\varepsilon \parens{U_{\varepsilon, z}}}_{\mathcal{D}^{1, 2}}^2,
\]
we obtain
\begin{eqnarray*}
\Dif_{U_{\varepsilon, z}} J_\varepsilon \brackets{\dot{U}_{\varepsilon, z, i}}
-
\varepsilon a \parens{\varepsilon z} \partial_i V \parens{\varepsilon z}
&=& 
\int
\Big(
		V_\varepsilon \parens{x}
		-
		V \parens{\varepsilon z}
		-
		\varepsilon \nabla V \parens{\varepsilon z} \cdot \parens{x - z}
	\Big)
	U_{\varepsilon, z} \parens{x}
	\dot{U}_{\varepsilon, z, i} \parens{x}
\dif x
\\
&+&
\varepsilon
\int
	\nabla V \parens{\varepsilon z} \cdot \parens{x - z}
	U_{\varepsilon, z} \parens{x}
\Big(
		\dot{U}_{\varepsilon, z, i} \parens{x}
		+
		\partial_i U_{\varepsilon, z} \parens{x}
	\Big)
\dif x
\\
&-&
\frac{\varepsilon}{2}
\left(
	2 \int
		\nabla V \parens{\varepsilon z} \cdot \parens{x - z}
		U_{\varepsilon, z} \parens{x}
		\partial_i U_{\varepsilon, z} \parens{x}
	\dif x
	+
	\partial_i V \parens{\varepsilon z} \norm{U_{\varepsilon, z}}_{L^2}^2
\right)
\\
&+&
\frac{3}{8}
\int 
	\Dif_{U_{\varepsilon, z}} \phi_\varepsilon \brackets{\dot{U}_{\varepsilon, z, i}}
	U_{\varepsilon, z}^2
+
\frac{3}{4}
\int
	\phi_\varepsilon \parens{U_{\varepsilon, z}} U_{\varepsilon, z} \dot{U}_{\varepsilon, z, i}
\\
&-&
\frac{1}{4 \varepsilon^2}
\angles*{
	\Dif_{U_{\varepsilon, z}} \phi_\varepsilon \brackets{\dot{U}_{\varepsilon, z, i}}
	~\middle|~
	\phi_\varepsilon \parens{U_{\varepsilon, z}}
}_{\mathcal{D}^{1, 2}}.
\end{eqnarray*}
Consider the terms on the right-hand side. In view of (\nameref{V_1}) and Lemma \ref{lem:tangent}, the terms on the first and second lines are of order $\varepsilon^2$. The term on the third line is zero. Indeed, an integration by parts shows that
\begin{multline*}
2 \int_0^\infty
	\nabla V \parens{\varepsilon z} \cdot \parens{x - z}
	U_{\varepsilon, z} \parens{x}
	\partial_i U_{\varepsilon, z} \parens{x}
\dif x_i
+
\partial_i V \parens{\varepsilon z}
\int_0^\infty
	U_{\varepsilon, z} \parens{x}^2
\dif x_i
=
\\
=
\int_0^\infty
	\nabla V \parens{\varepsilon z} \cdot \parens{x - z}
	U_{\varepsilon, z} \parens{x}^2
\dif x_i
=
0
\end{multline*}
because $x_i \mapsto U_{\varepsilon, z} \parens{x}$ is even and $x_i \mapsto \nabla V \parens{\varepsilon z} \cdot \parens{x - z}$ is odd. In view of Lemma \ref{lem:bound_phi_eps} and Corollary \ref{cor:estimate_Dif_phi_eps}, the terms on the fourth and fifth lines are of order $\varepsilon^2$.
\end{proof}

We can finally prove our main result by arguing similarly as in \cite[Proof of Theorem 8.5]{AM06}.

\begin{proof}[Proof of Theorem \ref{thm:mult}]

\emph{Multiplicity of solutions.}
In view of Lemma \ref{lem:expansion}, the result follows from an application of Theorem \ref{thm:approx_C1} by considering $f := C_0 V^\theta$; $g := \widetilde{J}_\varepsilon \parens{\cdot / \varepsilon}$; $\mathcal{M} := M$ and letting $\mathcal{N}$ be a bounded neighborhood of $\mathcal{M}$ in $\real^3$.

\emph{Concentration.}
We know that $\ker \Dif_x^2 V = \Tan_x M$ for every $x \in M$, so, up to shrinking $\mathcal{N}$, we can suppose that if $x \in \overline{\mathcal{N}}$ and $\nabla V\parens{x}=0$, then $x \in M$. Due to the previous result, we can fix $\varepsilon_0 \in \ooi{0, 1}$ and a family
\[
\set*{
	U_{\varepsilon, z_\varepsilon}+w_{\varepsilon, z_\varepsilon} =: u_\varepsilon \in H^1
	\mid
	\nabla J_\varepsilon \parens{u_\varepsilon} = 0
}_{\varepsilon \in \ooi{0, \varepsilon_0}}.
\]
It follows from Lemma \ref{lem:soln_aux} that
$\norm{u_\varepsilon - U_{\varepsilon, {z_\varepsilon}}}_{H^1} \to 0$
as $\varepsilon \to 0^+$. Suppose that $x_0 \in \overline{\mathcal{N}}$ is an accumulation point of $\set{z_\varepsilon}_{\varepsilon \in \ooi{0, \varepsilon_0}}$. There exists $\set{\varepsilon_n}_{n \in \nat} \subset \ooi{0, \varepsilon_0}$ such that $z_{\varepsilon_n} \to x_0$ as $n \to \infty$, so it follows from Proposition \ref{prop:conc} that $x_0 \in M$, hence the result.
\end{proof}

\medskip

\subsection*{Acknowledgment and statements}

This study was financed in part by the Coordenação de Aperfeiçoamento de Pessoal de Nível Superior - Brasil (CAPES) - Finance Code 001. G. de Paula Ramos was supported by Capes. G. Siciliano was supported by Capes, CNPq, FAPDF Edital 04/2021 - Demanda Espont\^anea, Fapesp grant no. 2022/16407-1 (Brazil) and Indam (Italy).

\smallskip

The authors declare that they have no conflict of interest. 

\smallskip

Data sharing not applicable to this article as no datasets were generated or analyses during the current study.


\begin{thebibliography}{99}



\bibitem{AAS}
N. Akhmediev, A. Ankiewicz, J. M. Soto-Crespo, 
{\sl  Does the nonlinear Schr\"odinger equation correctly describe beam equation? }
Opt. Lett. 18, (1993), 411--413.


\bibitem{ABC1997}A. Ambrosetti, M. Badiale, S. Cingolani, 
 {\sl Semiclassical states of nonlinear {Schr{\"o}dinger} equations, }
Arch. Ration. Mech. Anal. {\bf 140} no. 3 (1997), 285--300.


\bibitem{AMS01}
A. Ambrosetti, A. Malchiodi, S. Secchi, 
{\sl Multiplicity results for some nonlinear {Schr{\"o}dinger} equations with potentials, }
Arch. Ration. Mech. Anal. {\bf 159} no. 3 (2001), 253--271.


\bibitem{AM06} A. Ambrosetti, A. Malchiodi,
{Perturbation methods and semilinear elliptic problems on $\mathbb R^n$,  }
Progress in Mathematics, 2006,  Birkh{\"a}user.





\bibitem{BK08} K. Benmlih, O.  Kavian, 
{\sl Existence and asymptotic behaviour of standing waves for quasilinear {Schr{\"o}dinger}-{Poisson} 
systems in $\mathbb R^3$, }
Ann. Inst. Henri Poincar{\'e}, Anal. Non Lin{\'e}aire {\bf 25} no. 3 (2008), 449--470.


\bibitem{B11} H. Brezis,
{ Functional analysis, {Sobolev} spaces and partial differential equations, }
2011,  New York, NY: Springer.



\bibitem{C1993} K.-C. Chang, 
{\sl Infinite dimensional Morse theory and multiple solution problems,}
Progress in Nonlinear Differential Equations and Their Applications, 1993, Birkh{\"a}user.

\bibitem{DLMZ} L. Ding, L. Li, Y.-J. Meng, C.-L. Zhuang,
{\sl  Existence and asymptotic behaviour of ground state solution for quasilinear Schr\"odinger-Poisson systems in 
$\mathbb R^{3}$, } Topol. Methods Nonlin. Anal. {\bf 47} (2016), 241--264.




\bibitem{FS2} G. M. Figueiredo, G. Siciliano, 
{\sl  quasilinear Schr\"odinger-Poisson system under an exponential critical nonlinearity: existence and asymptotic behaviour of solutions, }
 Arch. Math. {\bf 112} (2019), 313--327.


\bibitem{FS20} G. M. Figueiredo, G. Siciliano, 
 {\sl Existence and asymptotic behaviour of solutions for a quasilinear {Schr{\"o}dinger}-{Poisson} system with a critical nonlinearity, }
Z. Angew. Math. Phys. {\bf 71} no. 130 (2020), 21pp.


\bibitem{GNN81} B. Gidas, W-M. Ni, L. Nirenberg,
{ \sl Symmetry of positive solutions of nonlinear elliptic equations in $\mathbb R^n$, }
Adv. {Math}., {Suppl}. {Stud}. 7A,  (1981), 369--402.



\bibitem{H} C. Hao,
{\sl The initial boundary value problem for quasilinear Schr\"odinger-Poisson equations, }
Acta Math. Sci. Ser. B (Engl. Ed.) {\bf 26} (2006), no.1, 115--124.


\bibitem{IKL} R. Illner, O. Kavian, H.  Lange, 
{\sl  Stationary solutions of quasilinear Schr\"odinger-Poisson system, }
 J. Diff. Equ. {\bf 145} (1998), 1--16.



\bibitem{ILT} R. Illner, H. Lange, B. Toomire, P.  Zweifel,
{\sl  On quasilinear Schr\"odinger-Poisson Systems, }
 Math. Methods in Applied Sciences, {\bf 20} (1997), 1223--1238.



\bibitem{MRS} P. A. Markovixh, C. Ringhofer, C.  Schmeiser,
 Semiconductor equations. Springer, Wien (1990).


\bibitem{IV08} I. Ianni, G. Vaira,
{\sl On concentration of positive bound states for the {Schr{\"o}dinger}-{Poisson} problem with potentials, }
Adv. Nonlinear Stud. {\bf 8} no.3 (2008), 573--595.


\bibitem{K89} M. K. Kwong,
{\sl Uniqueness of positive solutions of $\Delta u-u+u^ p=0 $ in $\mathbb R^ n$, } 
Arch. Ration. Mech. Anal. {\bf 105} no. 3 (1989), 243--266.


\bibitem{PJH} X. Peng, G. Jia, C. Huang, 
{\sl Quasilinear Schr\"odinger-Poisson system with exponential and logarithmic nonlinearities, }
Math. Methods Appl. Sci. {\bf 45} (2022), no.12, 7538--7554.



\bibitem{WLZ} C.  Wei, A. Li, L. Zhao, 
{\sl Multiple solutions for a class of quasilinear Sch\"odinger-Poisson system in $\mathbb R^{3}$
 with critical nonlinearity and zero mass, }
Anal. Math. Phys.{\bf 12} (2022), no.5, Paper No. 120, 20 pp.


\bibitem{WLZ2} C.  Wei, A. Li, L. Zhao, 
{\sl Existence and asymptotic behaviour of solutions for a quasilinear Schr\"odinger-Poisson system in
$\mathbb R^{3}$, }
Qual. Theory Dyn. Syst. {\bf 21} (2022), no.3, Paper No. 82, 15 pp.



\bibitem{LY} B. Li, H. Yang, 
{\sl The modified quantum Wigner system in weighted L2-space, }
 Bull. Aust. Math. Soc. {\bf 95} (2017), 73--83.

%

\end{thebibliography}
\end{document}